\documentclass[11pt]{amsart}
\usepackage{amsmath}
\usepackage{amsthm}
\usepackage{amssymb}
\usepackage{amsfonts,mathrsfs}
\usepackage{stmaryrd}
\usepackage{amsxtra}  
\usepackage{epsfig}
\usepackage{verbatim}
\usepackage[all]{xy}
\usepackage{hyperref}
\hypersetup{colorlinks=true, linkcolor = blue, citecolor = red}
\usepackage{bm}
\usepackage{tikz}

\theoremstyle{plain}
\newtheorem{theorem}[equation]{Theorem}
\newtheorem{proposition}[equation]{Proposition}
\newtheorem{lemma}[equation]{Lemma}
\newtheorem{corollary}[equation]{Corollary}
\newtheorem{definition}[equation]{Definition}
\theoremstyle{remark}
\newtheorem{remark}[equation]{Remark}
\numberwithin{equation}{section}
\newtheorem{example}[equation]{Example}

\newcommand{\dbar}{\bar \partial}

\def\norm#1{\Vert#1\Vert}

\newcommand{\cb}{{\mathcal B}}

\newcommand{\co}{{\mathcal O}}

\newcommand{\cs}{{\mathcal S}}

\newcommand{\sB}{{\mathscr B}}
\newcommand{\sC}{{\mathscr C}}
\newcommand{\sd}{{\mathscr D}}

\newcommand{\ve}{\varepsilon}

\newcommand{\C}{{\mathbb C}}
\newcommand{\D}{{\mathbb D}}

\newcommand{\Z}{{\mathbb Z}}

\begin{document}

\title[$\dbar$ equation]{Product domains, Multi-Cauchy transforms, \\ and the $\dbar$ equation}
\author{L. Chen \& J. D. McNeal}
\begin{abstract} Solution operators for the equation $\dbar u=f$ are constructed on general product domains in $\C^n$. When the factors are one-dimensional, the operator is a simple integral operator: it 
involves specific derivatives 
of $f$ integrated against iterated Cauchy kernels. For higher dimensional factors, the solution is constructed by solving sub-$\dbar$ equations  with modified data on the factors. Estimates of the operators in several norms are proved.
 \end{abstract}
\address{Department of Mathematics, \newline The Ohio State University, Columbus, Ohio, USA}
\email{chenliwei@wustl.edu}
\address{Department of Mathematics, \newline The Ohio State University, Columbus, Ohio, USA}
\email{mcneal@math.ohio-state.edu}

\maketitle 


\section*{Introduction}\label{S:intro}

In one complex variable, the Cauchy-Riemann equation can always be solved by explicit integral formulas. In particular if $D\subset\C^1$ is a domain with piecewise smooth boundary $bD$ and $f\in C(D)$,
then
\begin{equation}\label{I:cauchy}
v(z)= -\frac 1{2\pi i}\int_D\frac{f(\zeta)}{\zeta -z}d\bar\zeta\wedge d\zeta
\end{equation}
satisfies $\frac{\partial v}{\partial\bar z}(z)=f(z)$ for $z\in D$\footnote{Regularity on $f$ and $bD$ can be relaxed.}. This is abbreviated $\dbar v=fd\bar z$  and called the $\dbar$-equation subsequently. 
The elementary theory of holomorphic functions in one variable can be based on \eqref{I:cauchy},  see \cite[Chapter 1]{hormander_scv_book}, \cite{nar_book}, \cite[Chapter 3]{berstein_gay}, rather than
the usual basis of the Cauchy integral formula. Deeper results also follow from the explicit form of  \eqref{I:cauchy}, e.g., a proof of the Corona theorem \cite{gamelin} and interpolation theorems \cite{jones83}. 

In several variables, there is no analogous universal solution operator for $\dbar$. An initial difficulty is that the $\dbar$-equations are over-determined in $\C^n$, $n>1$: solving $\dbar u=f$,  for $f$ a $(0,1)$-form, requires restricting to forms satisfying $\dbar f=0$. A second difficulty is that integral formulas solving $\dbar$ only exist on special classes of domains $\Omega\subset\C^n$. When they exist, the kernels of these integral operators depend on the geometry  of 
$b\Omega$ and are complicated. Domain dependence of the kernels means showing boundedness of the associated operators on normed spaces  (e.g. $L^p$ and $\Lambda^\alpha$)  requires individual analysis. 

The purpose of this paper is to establish comparatively simple formulas solving $\dbar$ on arbitrary product domains in $\C^n$. We also show these operators are bounded on modifications of $L^p$ and H\" older spaces and standard $L^2$ Sobolev spaces; when one of the factors has dimension $>1$, commutativity assumptions are needed to establish boundedness. The formulas are based on
\eqref{I:cauchy} when the factors are one-dimensional. For higher dimensional factors, $\Omega=D_1\times\dots\times D_k$ with $D_j\subset\C^{n_j}$, the formulas for $\dbar u=f$ on $\Omega$ involve solutions to sub-$\dbar$ equations
on the $D_j$, for data involving components of $f$ and its derivatives.

For products with one-dimensional factors, $D=D_1\times\dots\times D_n$ with $D_j\subset\C$, the solution operator is inspired by our previous Proposition 2.2 in \cite{CheMcN18}. Let $f=\sum_{j=1}^n f_jd\bar z_j$ satisfy $\dbar f=0$ on $D$. If $J=\{j_1,\dots,j_l\}\subset\{1,\dots ,n\}$, set 
\begin{equation*}
f_{J}=\frac{\partial^{l-1}f_{j_1}}{\partial\bar\zeta_{j_2}\cdots\partial\bar\zeta_{j_l}}\quad\text{ and }\quad f^J= f, \text{ but with the variables } z_{j_1},\dots z_{j_l} \text{ fixed}. 
\end{equation*}
Let $\bm{C}^J$ denote the partial solid Cauchy transform, see \eqref{E:solidmultiCT}. Define 
\begin{equation}\label{D:Tintro}
T(f)=-\sum_{\emptyset\neq I\subset\{1,\dots,n\}}\bm{C}^{I}(f_{I}^{I^c}),
\end{equation}
the sum taken over all non-empty subsets and $I^c=\{1,\dots ,n\}\setminus I$. The central result is
\begin{theorem}\label{T:intro} 
If $f\in\cb$ (see Definition \ref{Banach}) then $\dbar \left(Tf\right)=f$ weakly.
\end{theorem}

This is proved as Theorem \ref{main1} below. It is remarkable the simple \eqref{D:Tintro} solves $\dbar$ on $D$; its simple derivation in Section 2 seems noteworthy as well. Known integral formulas solving $\dbar$ on strongly pseudoconvex domains are much more intricate -- see especially Section 2 in \cite{RangeSiu}, or the various formulas in \cite{Range86, HenkinLeiterer}. 

On the other hand, derivatives of the data form $f$ occur in  \eqref{D:Tintro}. When compared to previously studied cases this appears restrictive. An initial aim of the paper is to show that allowing derivatives leads
to short, symmetric expressions like  \eqref{D:Tintro}. A deeper aim concerns estimates. $L^p$ estimates on  $T$ are proved in Section \ref{S:Lp} and H\" older estimates, of a non-standard kind, are proved in Section \ref{S:Holder}.
One result is

\begin{theorem}\label{C:LpIntro}
Let $p\in[1,\infty]$. Suppose $f_{I}\in L^p(D)$ for all $I\neq\emptyset$ and $\dbar f=0$. Then 
\[
\|T(f)\|_{L^p(D)}\le C\sum_{I\neq\emptyset}\|f_{I}\|_{L^p(D)}.
\]
\end{theorem} 
\noindent The bounds obtained are natural. The proofs of Propositions \ref{T:Lp}, \ref{highdimexistence}, \ref{LpforhighT}, and \ref{main3} show how an accumulation of estimates on $f_I$ are needed to control $Tf$. 
Subsection \ref{SS:example} also shows that $f\in L^p_{(0,1)}(D)$ alone is not sufficient to conclude that $Tf\in L^p(D)$, for certain $p$. Applying Stokes' theorem to  \eqref{D:Tintro} (multiple times) leads to a re-expression of $T$ without derivatives on $f$, but at the
cost of many partial boundary integrals on the lower dimensional facets of $bD$; see Remark 2.3 in  \cite{CheMcN18}.

An additional remark about derivatives in  \eqref{D:Tintro}: note that only barred derivatives occur in $f_I$ and only certain barred derivatives land on individual components of $f$. In particular,
$\sum_{I\neq\emptyset}\|f_{I}\|_{L^p(D)}$ is weaker than the full Sobolev norm $\|f\|_{W^{n-1, p}_{(0,1)}(D)}$.

Turning to higher dimensional factors, $D=D_1\times\dots\times D_k$ with $D_j\subset\C^{n_j}$, the strategy is to construct a solution operator $T$ on $D$ from given solution operators $T_j$ on the factors $D_j$.
This requires that the factor decomposition of $D$ be analytically respected. The first step is to decompose $\dbar=\dbar_1+\dots +\dbar_k$, where $\dbar_j$ is the Cauchy-Riemann operator on $D_j$, and extend 
$f_J$ to separately account for derivatives in the variables associated to $D_j$; see Definition \ref{fI}. Next, interactions between the operators $T_j$ and derivatives in the {\it complementary factors} to $D_j$ arise. 
Generically denote barred derivatives associated to $D_1,\dots,D_{j-1},D_{j+1},\dots,D_k$ by $\partial/\partial\bm{\bar z}_*$.
The following hypothesis is used.
\begin{itemize}
\item[(C)] For each $j=1,\dots , k$, there is a linear bounded operator $T_j:L^p_{(0,1)}(D_j)\to L^p(D_j)$ that solves the $\dbar_j$-equation on $D_j$ and commutes with all $\partial/\partial\bm{\bar z}_*$ .
\end{itemize}
 The main $L^p$ result for higher dimensional factors is
\begin{theorem}\label{P:existenceIntro}
Let $D=D_1\times\cdots\times D_k$, with $D_j\subset\C^{n_j}$, and $p\in[1,\infty]$.  Assume (C) holds.
Then there is a linear operator $T$ on
\[
\sB^p:=\{f\,\dbar\text{-closed}\,|\,f_I\in L^p(D)\text{ for all }I\neq0\}
\]
satisfying $\dbar T(f)=f$. Furthermore the estimate
\[
\|T(f)\|_{L^p(D)}\le C\sum_{|I|\neq0}\|f_I\|_{L^p(D)}
\]
holds, for $C>0$ independent of $f$.
\end{theorem}

This is contained in Propositions \ref{highdimexistence} and \ref{LpforhighT}. The proof of Theorem \ref{P:existenceIntro} involves solving $\dbar_j$ equations on the factors $D_j$, where the $\dbar_j$ data comes from
derivatives of $f$ on $D_1,\dots,D_{j-1},D_{j+1},\dots,D_k$. In particular, the $\dbar_j$ datum are not merely the restriction of $f$ to the individual factors $D_j$. See Remark \ref{R:explainC}. The same basic method yields
existence and regularity results on $L^2$ Sobolev spaces and non-standard H\" older spaces, as shown in Sections \ref{smoothregularity}-\ref{S:Holder}.

Hypothesis (C) is not very restrictive. For instance if the operators $T_j$ on $D_j$ are integral solution operators, (C) holds; see Lemma \ref{L:IntegralsCommute}. Thus assuming (C) was
unnecessary when the factors are one-dimensional. More generally,  when each $D_j$ is strongly pseudoconvex it is unnecessary to assume (C):  the $L^p$ bounded solutions in \cite{Kerzman71, Lieb70, Ovrelid} are integral operators on $D_j$.  Hypothesis (C) also holds when natural projection operators on the factors $D_j$ exist. Let $A^p(D_1)= L^p(D_1)\cap \co\left(D_1\right)$,
$\co(D_1)$ denoting holomorphic functions on $D_1$. 

\begin{proposition} Let $D=D_1\times\cdots\times D_k$, $p\in (1,\infty)$, and suppose $T:L_{(0,1)}^p(D_1)\to L^p(D_1)$ solves the $\dbar_1$-equation on $D_1$. 
Assume there exists $P:L^p(D_1)\to A^p(D_1)$ which preserves $A^p(D_1)$. 

Then $S=T-P\circ T$ satisfies $\big[\partial/\partial\bm{\bar z}_*,S\big]=0$, where $\partial/\partial\bm{\bar z}_*$ denotes an arbitrary barred derivative in the variables on
$D_2\times\dots\times D_k$.
\end{proposition}
See Lemma \ref{commutativity} for a more precise statement. Notice that $S:L_{(0,1)}^p(D_1)\to L^p(D_1)$ and $\dbar_1\left(Sf\right)=f$, so hypothesis (C) is satisfied if suitable projection operators $P$ exist on each factor $D_j$.
The Bergman projection $\bm{B}$ on $D_j$ is a natural candidate for such $P$. However, preservation of $A^p(\Omega)$ for $p\neq 2$ is not automatic; see \cite{ChaEdhMcN18} for examples of pseudoconvex $\Omega$
where $\bm{B}_\Omega$ does not preserve $A^p(\Omega)$.

In Section \ref{S:ortho} orthogonality properties of the $\dbar$ solution operators  constructed in the paper are briefly discussed.

Integral solution operators for $\dbar$ have been constructed on various classes of domains, starting with the breakthrough work of
\cite{Henkin69} and \cite{GrauertLieb} on strongly pseudoconvex domains, and shown to be bounded on various classical normed spaces.
There are many significant results in this direction, too numerous to adequately summarize here. The foundational papers \cite{GrauertLieb, Henkin69, Kerzman71, Lieb70, Ovrelid, RangeSiu},  the books \cite{Range86, HenkinLeiterer}, and the paper \cite{LanSte13} present the principal theorems. The bibliographies in \cite{HenkinLeiterer, LiebMichelbook, Range86} give references to more specialized results. The papers \cite{Fornaess86} and \cite{MichelShaw99}
are highlights of results obtained after the mid-80s.

Theorem \ref{C:LpIntro} contains the case $p=\infty$ and holds in $\C^2$. This case received special attention previously.
An integral formula solving $\dbar$ on the bidisc $\D^2$ was stated in \cite{Henkin71}. Estimates in $L^\infty$ on this operator are shown in \cite{FornaessLeeZhang}, when the data is sufficiently smooth, i.e. $f\in C^1\left(\overline{\D^2}\right)$. Specializing a complicated integral formula on polyhedra to $\D^2$,  $L^\infty$ estimates on a solution to $\dbar u=f$ were stated in \cite{HenkinSergeev}, again if $f$ is sufficiently smooth on $\overline{\D^2}$.
Some details of these estimates appear in  \cite{HenkinSergeev}, others in \cite{Jakobczak89}.
In \cite{CheMcN18}, the operator  \eqref{D:Tintro} is introduced on general product domains in $\C^2$ and the following shown: for less regular $f$ than assumed above  -- namely if only $\frac{\partial f_1}{\partial\bar z_2}\in L^\infty\left({\D^2}\right)$ where $f=f_1d\bar z_1+f_2 d\bar z_2$ -- the $L^\infty$ estimates in \cite{FornaessLeeZhang} hold.  

Nevertheless, the results in 
\cite{HenkinSergeev, FornaessLeeZhang, CheMcN18} and Theorem \ref{C:LpIntro} above leave open a question
posed in \cite{Kerzman71}: if $f\in L^\infty_{(0,1)}\left(\D^2\right)$ is weakly $\dbar$-closed, is there a $u\in L^\infty\left(\D^2\right)$ solving $\dbar u=f$?
The issue is passage from a priori $L^\infty$ estimates to genuine estimates. Within the $L^p$ scale this issue is special to $L^\infty$: a general
$g\in L^\infty$ cannot be approximated in $L^\infty$ norm by convolution with smooth bump functions while $g\in L^p$ for $p<\infty$ can be $L^p$-approximated by such convolutions. 

In a different direction, $L^2$ estimates for $\dbar$ on product domains are proved in \cite{ChakrabartiShaw11}, without first establishing integral solution operators. These estimates hold on $\D^2$ without additional assumptions on $f$ beyond $f\in L^2_{(0,1)}\left(\D^2\right)$.\footnote{After this paper was written, the authors received \cite{FassinaPan}, where
 integral formulas on product domains with one-dimensional factors are also obtained.}



\section{The basic formula; smooth data}\label{S:integral}

\subsection{One-dimensional Cauchy transforms}\label{SS:1dCT} One variable Cauchy transforms underlie our initial constructions. To focus on
the new several variable phenomena, needed results about these transforms are deferred to an Appendix, Section \ref{S:appendix}.

Let $D_0\subset\C$ be a bounded domain with  piecewise $C^1$ boundary $bD_0$. If $g\in C(bD_0)$, the {\it Cauchy transform of} $g$ is defined
\begin{equation}\label{E:CT}
\sC(g)(z)=\frac{1}{2\pi i}\int_{bD_0}\frac{g(\zeta)\,d\zeta}{\zeta-z}.
\end{equation}
Differentiation under the integral shows $\sC g\in\co(D_0)$. If $h\in C(D_0)$, the {\it solid Cauchy transform of} $h$ is defined
\begin{equation}\label{E:solidCT}
\bm{C}(h)(z)=\frac{1}{2\pi i}\int_{D_0}\frac{h(\zeta)\,d\bar\zeta\wedge d\zeta}{\zeta-z}\qquad h\in C(D_0).
\end{equation}

Smoothness of $\bm{C}(h)$ is less obvious than $\sC(g)$, since the singularity at $\zeta=z$ occurs inside the region of integration; however see Lemma \ref{L:smoothCT}.

\subsection{Products with one-dimensional factors}

Let $D=D_1\times\cdots\times D_n$ be a bounded product domain in $\C^n$, where  $D_j\subset\C$ with piecewise $C^1$ boundary for $j=1,\dots,n$. Extend \eqref{E:CT}
to $D$ as the \emph{multi-Cauchy transform}
\[
\sC_n:C(bD_1\times\cdots\times bD_n)\to\co(D),
\]
by defining
\begin{equation}\label{E:multiCT}
\sC_n(g)(z)=\frac{1}{(2\pi i)^n}\int_{bD_1\times\cdots\times bD_n}\frac{g(\zeta_1,\dots,\zeta_n)\,d\zeta_1\wedge\cdots\wedge\,d\zeta_n}{(\zeta_1-z_1)\cdots(\zeta_n-z_n)}
\end{equation}
for $g\in C(bD_1\times\cdots\times bD_n)$.

Consider the $\dbar$ equation on $D$, $\dbar u=f$,
where $f$ is a $(0,1)$-form with $\dbar f=0$. Write
$f=f_1d\bar z_1+\cdots +f_nd\bar z_n$, where each component $f_j=f_j(z_1,\dots ,z_n)$ is a function at least in $C^1(D)$.
{\it Assume}, for the remainder of this section, that there is a solution $u\in C^n(\overline{D})$ and that $f\in C^{n-1}_{(0,1)}(\overline{D})$.

Since $u\in C(\overline{D})$,  the integral giving $\sC_n u$ is well-defined; differentiating under the integral shows $\sC_n(u)\in\co(D)$. Fubini's theorem implies  $\sC_n u$ can be expressed iteratively 
\begin{equation}\label{E:fubini}
\sC_n(u)(z)=\sC_{n-1}\big(\sC(u(\cdot,z_n))\big)(z'),
\end{equation}
where $z'=(z_1,\dots,z_{n-1})$, the inner $\sC$ operator acts on the $z_n$ variable according to \eqref{E:CT}, and the outer $\sC_{n-1}$ operator acts on the $\cdot$ variables according to \eqref{E:multiCT}.
Other expressions for $\sC_n u$ also hold -- taking any of the variables $z_k$ in place of $z_n$ -- since the order of integration in $\sC_n u$ is irrelevant.

Apply Lemma \ref{L:stokes} to the inner operator $\sC$ in \eqref{E:fubini}, i.e. on the variable $z_n\in D_n$. This yields the recursive formula 
\begin{equation}
\label{recursive}
\sC_n(u)(z)=\sC_{n-1}\big(\bm{C}(u_{\bar\zeta_n})(\cdot,z_n)\big)(z')+\sC_{n-1}\big(u(\cdot,z_n)\big)(z'),
\end{equation}
where $z'=(z_1,\dots,z_{n-1})$. Other orders of this iterative expression are valid; any equivalent expression will be referred to as \eqref{recursive}.

\subsubsection{Notation for derivatives and evaluation}

To exploit \eqref{recursive}, notation involving subsets of $\{1,2,\dots,n\}$ is introduced. If $I\subset\{1,2,\dots,n\}$ with $0\le|I|=l\le n$, $I$ can be expressed uniquely as $I=\{i_1,\dots,i_l\}$ with $i_1 <\dots <i_l$. 
In all definitions below, the presentation of $I$ is irrelevant. When ``$I\subset\{1,2,\dots,n\}$'' appears as a summation index, the sum is taken only {\it once} for each $I$, e.g. for $I$ expressed in increasing form.

Subscripts will denote partial derivatives
 and superscripts will denote partial evaluation. This will be done differently for functions and $(0,1)$-forms, to expose the basic new relations on the solution operator. 
 
If $u\in C^n(\overline{D})$ and $\emptyset\neq I\subset\{1,2,\dots,n\}$, define
\begin{equation}\label{E:subI}
u_{I}=\frac{\partial^{l}u}{\partial\bar\zeta_{i_1}\cdots\partial\bar\zeta_{i_l}}
\end{equation}
and
\begin{equation}\label{E:superI}
u^{I}=u(\cdots,z_{i_1},\cdots,z_{i_2},\cdots,z_{i_l},\cdots)\qquad\text{i.e., evaluation in the slots } z_{i_1},\dots,z_{i_l},
\end{equation}
with the remaining variables free. If $I=\emptyset$, simply let
$u_{\emptyset}= u^{\emptyset}=u$.

The notation for $(0,1)$-forms is similar, but shifted. Let $f=\sum_j f_j\, d\bar z_j\in C^{n-1}_{(0,1)}(\overline{D})$ and consider first a singleton $I\subset\{1,2,\dots,n\}$. Define
\begin{equation*}
f_{I}=f_j,\qquad \text{when }I=\{j\}\subset\{1,2,\dots,n\}.
\end{equation*}
If $I=\{i_1,\dots,i_l\}$ and $l>1$, let
\begin{equation}\label{E:subl_form1}
f_{I}=\frac{\partial^{l-1}f_{i_1}}{\partial\bar\zeta_{i_2}\cdots\partial\bar\zeta_{i_l}}.
\end{equation}
No meaning is assigned to $f_{\emptyset}$ when $f$ is a $(0,1)$-form. In all cases, for a $(0,1)$-form $f$, $f_{I}$ involves barred derivatives of order $l-1$, while  $u_{I}$ involves barred derivatives of order $l$ when $u$ is a function. Note the association $f\longrightarrow f_{I}$ is a mapping from $(0,1)$-forms to functions.

Now suppose that $\dbar u=f$. When $I=\{j\}$, it follows immediately that $f_{I}=f_j=\frac{\partial u}{\partial\bar z_j}= u_{I}$. Furthermore, since $\dbar f=0$, it follows that $f_{I}$ depends only on the elements of $I\subset\{1,2,\dots,n\}$ and not on its presentation, e.g. which element is designated $i_1$. Differentiating \eqref{E:subI} and \eqref{E:subl_form1} thus gives
\begin{equation}\label{E:fI=uI}
f_{I}=u_{I}\qquad\text{when }\dbar u=f,
\end{equation}
for all $\emptyset\neq I\subset\{1,2,\dots,n\}$.

\subsubsection{A representation result}

For any $\emptyset\neq I\subset\{1,\dots ,n\}$, define the  {\it (partial) solid multi-Cauchy transform} of $u\in C(D)$ as
\begin{equation}\label{E:solidmultiCT}
\bm{C}^{I}(u)=\frac{1}{(2\pi i)^l}\int_{D_{i_1}\times\cdots\times D_{i_l}}\frac{u\,d\bar\zeta_{i_1}\wedge d\zeta_{i_1}\wedge\cdots\wedge d\bar\zeta_{i_l}\wedge d\zeta_{i_l}}{(\zeta_{i_1}-z_{i_1})\cdots(\zeta_{i_l}-z_{i_l})}.
\end{equation}
If $I=\emptyset$, let
$ \bm{C}^{\emptyset}(u)=u$.

The subscript/superscript notation yields a succinct extension of Lemma \ref{L:stokes} to product domains.

\begin{lemma}
\label{indCauchy}
Let $D=D_1\times\cdots\times D_n$ be a bounded product domain in $\C^n$, with each  $D_j\subset\C$ having piecewise $C^1$ boundary. If $u\in C^n\left(\overline{D}\right)$, then
\begin{equation}\label{E:repCauchySolid}
\sC_n(u)=\sum_{I\subset\{1,2,\dots,n\}}\bm{C}^{I}\left(u_{I}^{I^c}\right),
\end{equation}
where $I^c$ is the complement of $I$ in $\{1,2,\dots,n\}$.
\end{lemma}

\begin{proof}
Induct on $n$. When $n=1$, Lemma \ref{L:stokes} gives the conclusion. Assume the conclusion holds for $n=1,\dots,k-1$ ($k\ge2$). 

For $n=k$, the formula \eqref{recursive} implies
\[
\sC_k(u)=\sC_{k-1}(\bm{C}^{\{k\}}u_{\{k\}})+\sC_{k-1}(u^{\{k\}}).
\]
Using the inductive assumption, this yields 
\begin{align*}
\sC_k(u)
&=\sum_{I\subset\{1,\dots,k-1\}}\bm{C}^{I}\left((\bm{C}^{\{k\}}u_{\{k\}})_{I}^{I^c}\right)+\sum_{I\subset\{1,\dots,k-1\}}\bm{C}^{I}\left((u^{\{k\}})_{I}^{I^c}\right)\\
&=\sum_{I\subset\{1,\dots,k-1\}}\bm{C}^{I\cup\{k\}}\left(u_{I\cup\{k\}}^{I^c}\right)+\sum_{I\subset\{1,\dots,k-1\}}\bm{C}^{I}\left(u^{I^c\cup\{k\}}_{I}\right)\\
&=\sum_{k\in I\subset\{1,\dots,k\}}\bm{C}^{I}\big(u_{I}^{I^c}\big)+\sum_{k\notin I\subset\{1,\dots,k\}}\bm{C}^{I}\big(u^{I^c}_{I}\big)\\
&=\sum_{I\subset\{1,\dots,k\}}\bm{C}^{I}(u_{I}^{I^c}).
\end{align*}
Thus the conclusion holds for $n=k$ as well.
\end{proof}

\begin{remark} Lemma \ref{indCauchy} has seemingly not been noted before. Once noticed, the result may also be derived by multiple applications of Stokes theorem, starting with the terms on the right-hand side of \eqref{E:repCauchySolid}. However these computations are quite cumbersome even when $n$ is small, in contrast to the
almost trivial proof given above.
\end{remark}

\subsubsection{New solution from old}
The interplay evaluation-differentiation-transform expressed by Lemma \ref{indCauchy} provides a simple and powerful method for constructing solutions to $\dbar$ on product domains. Recall that
 $\bm{C}^{\emptyset}(u_{\emptyset}^{\emptyset^c})=u(z)$. The formula \eqref{E:repCauchySolid} can therefore be rewritten as
\begin{equation}
\label{CIu}
u-\sC_n(u)=-\sum_{\emptyset\neq I\subset\{1,\dots,n\}}\bm{C}^{I}\left(u_{I}^{I^c}\right).
\end{equation}
Similar to Corollary \ref{L:C1solution}, \eqref{CIu} gives a solution operator to the $\dbar$-equation on $D$. 
\begin{proposition}\label{integralformula}
Let $D=D_1\times\cdots\times D_n$ be a bounded product domain in $\C^n$, where $D_j\subset\C$ are domains with piecewise $C^1$ boundaries. 

Suppose $u\in C^n(\overline{D})$ solves the equation $\dbar u=f$, where $f\in C^{n-1}_{(0,1)}(\overline{D})$ is $\dbar$-closed. Then 
\begin{equation}\label{E:strongsolution}
v(z)=: -\sum_{\emptyset\neq I\subset\{1,\dots,n\}}\bm{C}^{I}\left(f_{I}^{I^c}\right)
\end{equation}
also satisfies $\dbar v=f$.
\end{proposition}

\begin{proof} Since $\dbar u=f$,  \eqref{CIu} becomes
\begin{align}\label{CIf}
u-\sC_n(u)&=-\sum_{\emptyset\neq I\subset\{1,\dots,n\}}\bm{C}^{I}\left(u_{I}^{I^c}\right) \notag \\
&=-\sum_{\emptyset\neq I\subset\{1,\dots,n\}}\bm{C}^{I}\left(f_{I}^{I^c}\right),
\end{align}
by \eqref{E:fI=uI}.

Take $\dbar$ of both sides of \eqref{CIf}: the result on the left-hand side is $f$, since $\dbar u=f$ and $\sC_n(u)$ is holomorphic. 
\end{proof}

Proposition \ref{integralformula} provides an explicit solution to $\dbar v=f$, but only under the restrictive
assumptions that (i) $f\in C^{n-1}_{(0,1)}(\overline{D})$ and (ii) a solution $u\in C^n(\overline{D})$ to $\dbar u=f$ is already known to exist. These auxiliary assumptions are
unnatural and limit application of Proposition \ref{integralformula}. Removing these restrictions is turned to next.


\section{$L^p$ data}
\label{S:Lp}

Let $D=D_1\times\cdots\times D_n$ be a bounded product domain in $\C^n$, with $D_j\subset\C$ domains with piecewise $C^1$ boundaries. 
Consider the $\dbar$-equation on $D$ with $L^p$ data:  $\dbar u=f$
where $f$ is a $\dbar$-closed $(0,1)$-form with $L^p$-coefficients, $1\le p\le \infty$, differentiations taken in the weak sense. 
Our goal is to establish existence and $L^p$ regularity of a solution to $\dbar$ under conditions naturally arising from the structure of the $L^p$ spaces.
Existence of the solution operator is derived first, to highlight the new boundedness assumptions. Estimates in $L^p(D)$ follow from standard arguments presented in
subsection \ref{SS:LpRegularity}.

\subsection{Existence}
A non-isotropic $L^p$-Sobolev norm occurs in the existence result. Notation from \S\ref{S:integral} is used and all derivatives are taken distributionally.
\begin{definition}\label{Banach}
On the space of $(0,1)$-forms on $D$, let
\[
\cb=\{f \,|\dbar f=0 (weakly) \text{ and }\,f_{I}\in L^1(D),\text{for } I\subset\{1,2,\dots,n\},\,I\neq\emptyset\}
\]
be the Banach space with norm
\[
\|f\|_{\cb}=\sum_{I\neq\emptyset}\left\|f_{I}\right\|_{L^1(D)}.
\]
\end{definition}

The solution operator will be given on $\cb$. 

\begin{definition}\label{D:T}
For $f\in\cb$, define

\begin{equation}\label{T}
T(f)=-\sum_{\emptyset\neq I\subset\{1,\dots,n\}}\bm{C}^{I}(f_{I}^{I^c}).
\end{equation}
\end{definition}
The integral operators on the right of \eqref{T} are not immediately well-defined.  For one reason,  ``evaluation'' -- implied by the superscript $I^c$ -- needs to be interpreted for  $L^p$ functions. Also,
 $\bm{C}^{I}$ has so far only been defined for
continuous functions. 

Both issues are readily handled, yielding the following:

\begin{proposition}\label{P:bounded} The operator  $T:\cb\to L^1(D)$ is norm bounded.
\end{proposition}

\begin{proof} If $f\in\cb$, then $f_{I}\in L^1(D)$ for all $I\subset\{1,\dots , n\}$. For fixed $I$, suppose $I=\{i_1,\dots ,i_m\}$ with $i_1<\dots < i_m$.  Fubini's
theorem implies 
\begin{equation}\label{E:fubini2}
\int_{D_{i_1}\times\cdots\times D_{i_m}} \left| f_{I}(\zeta)\right|\, d\bar\zeta_{i_1}\wedge d\zeta_{i_1}\wedge\cdots\wedge d\bar\zeta_{i_m}\wedge d\zeta_{i_m}\in L^1 \left(D^c\right),
\end{equation}
where $D^c$ is the complementary product space $D_{j_1}\times\cdots\times D_{j_{n-m}}$ with $\{i_1,\dots ,i_m\}\cup\{j_1,\dots j_{n-m}\} =\{1,2,\dots ,n\}$.
Fubini' theorem also implies that $f_{I}(\cdot,\dots, z_{j_1},\dots, z_{j_{n-m}}, \dots, \cdot)\in L^1\left(D_{i_1}\times\cdots\times D_{i_m}\right)$ for a.e. $(z_{j_1},\dots ,z_{j_{n-m}})\in D_{j_1}\times\cdots\times D_{j_{n-m}}$.
Define $f_{I}^{I^c}$ to be this $L^1$ function.

Consider \eqref{T}: the operators $\bm{C}^{I}$ are partial solid Cauchy transforms over various sub-product domains obtained by intersecting $D$ with coordinate hyperplanes. These integrals can be evaluated as iterated one-variable integrals, again by Fubini.
Repeatedly applying Lemma \ref{Lp1dim} -- for different $g=f_{I}^{I^c}$, as $I$ varies over subsets of $\{1,\dots , n\}$ and $j$ ranges between $1$ and $n$ -- gives the claimed result.
\end{proof}

The operator $T$ is thus defined on $\cb$ and bounded into $L^1(D)$.
A separate argument is needed to show \eqref{T} solves $\dbar$:

\begin{theorem}
\label{main1} Let $D=D_1\times\cdots\times D_n$ be a bounded product domain in $\C^n$, where $D_j\subset\C$ are domains with piecewise $C^1$ boundaries. 

For $f\in\cb$, the operator $T$ in \eqref{T} is a weak solution operator for the $\dbar$-equation:
\[
\left(\dbar (Tf),\varphi\right) =(f, \varphi) \quad\forall\varphi\in C^\infty_0(D),\quad \text{ if } f\in\cb.
\]
\end{theorem}

\begin{proof} It suffices to assume each $D_j$ as $C^1$ bounded, since $D_j$ as in the hypothesis can be exhausted by such domains.
For $j=1,2,\dots,n$, let $\rho_j$ be a defining function of $D_j$. For $\delta>0$, let
\[
D^{\delta}_j=\{z\in\C\,|\,\rho_j(z)<-\delta\}
\]
and $D^{\delta}=D_1^{\delta}\times\cdots\times D_n^{\delta}$. Let $\cb^{\delta}$ be the Banach space defined in Definition \ref{Banach} with $D$ replaced by $D^{\delta}$. Let $T^{\delta}$ be the operator defined  in  \ref{T} with $\cb$ and $D$ replaced by $\cb^{\delta}$ and $D^{\delta}$ respectively.

Temporarily fix $\delta >0$.
For $f\in\cb$, convolution with an approximate identity (see e.g. \cite[Chap. 5.3, Theorem 1]{Evans98}) gives a sequence $\{f^{\ve}\}\subset  C^{\infty}_{(0,1)}(D^{\delta/2})\subset C^{\infty}_{(0,1)}(\overline{D^{\delta}})$ satisfying
\[
\dbar f^{\ve}=0,
\]
\[
f^{\ve}\to f\,\,\,\,\text{and}\,\,\,\,f_{I}^{\ve}\to f_{I}\,\,\,\,\text{in}\,\,L^1\left(D^{\delta}\right)\,\,\,\,\text{as}\,\,\ve\to0^+,
\]
for all $ I\subset\{1,2,\dots,n\},\,I\neq\emptyset$. Note the derivatives implicit in the notation $f_{I}$ are all constant-coefficient differential operators, so the claimed approximation is indeed the standard mollifier argument.

Thus  $f^{\ve}\to f$ in $\cb^{\delta}$ as $\ve\to0^+$ and $\{f^{\ve}\}$ are all $\dbar$-closed on $D^{\delta/2}$. 
Since $D^{\delta/2}$ is pseudoconvex,  Hormander's theorem \cite{hormander_scv_book} implies the existence of $u^{\ve}\in L^2\left(D^{\delta/2}\right)$ solving $\dbar u^{\ve}=f^{\ve}$ for all ${\ve}$ sufficiently small. Since 
$f^{\ve}\in C^{\infty}_{(0,1)}(D^{\delta/2})$, interior regularity
of $\dbar$ on functions implies that $u^{\ve} \in C^{\infty}(D^{\delta/2})\subset C^{\infty}(\overline{D^{\delta}})$. Consequently the hypotheses of Proposition \ref{integralformula} hold, so $v^{\delta,\ve}=T^{\delta}(f^{\ve})$ solves $\dbar v^{\delta,\ve}=f^{\ve}$ on $D^{\delta}$.

Now let $\ve\to 0^+$. The fact that $f^{\ve}\to f$ in $\cb^{\delta}$ and Proposition \ref{P:bounded} (applied to $T^{\delta}$) implies $u=T^{\delta}(f)$ weakly solves $\dbar u=f$ on $D^{\delta}$. But Proposition \ref{P:bounded} also implies
\[
\lim_{\delta\to0^+}T^{\delta}(f)=T(f)
\]
for each $f\in\cb$. Letting $\delta\to 0^+$ shows $\dbar\left(T(f)\right)=f$ weakly, as claimed.
\end{proof}

\begin{remark} Interior regularity for $\dbar$ on functions is used in the proof of Theorem \ref{main1}. This  fails for higher level forms; an extension of
Theorem \ref{main1} to data in $L^p_{0,q}$, $q>1$, must deal with this fact.  See, e.g., \cite{McnVar17} for an approach that circumvents this difficulty in another context.
\end{remark}

\subsection{$f\in L^p_{(0,1)}(D)\not\Rightarrow$ existence of $Tf$}\label{SS:example}
The following example motivates why $T$ is restricted to $\cb$ in Proposition \ref{P:bounded}. The example in particular  shows requiring  $f\in L^p_{(0,1)}(D)$ alone does
not guarantee existence of $Tf$.

\begin{example} Let $\D^n$ denote the unit polydisc. Given $I\subset\{1,2,\dots,n\}$ and $I\neq\emptyset$. Assume $I=\{i_1,\dots,i_l\}$ and let
\[
z_I=z_{i_1}\cdots z_{i_l}.
\]
We construct a $\dbar$-closed $(0,1)$ form $f$, such that $f_{I}\notin L^1(\D^n)$ and $f_{J}\in L^1(\D^n)$ for all $J\subset\{1,2,\dots,n\}$, $J\neq\emptyset$, and $J\neq I$, but $T(f)$ does not exist. Moreover, when $|I|=l>1$, such $f$ is actually in $L^p_{(0,1)}(\D^n)$ for $1\le p<l$.

For $k=1,2,\dots$, on $\D^n$ let
\[
u^k(z):=\frac{1}{k}|z_I|^{2k}
\]
and
\[
f^{k}=\dbar u^k=z_{i_1}|z_{i_1}|^{2k-2}|z_{i_2}\cdots z_{i_l}|^{2k}\,d\bar z_{i_1}+\cdots+z_{i_l}|z_{i_l}|^{2k-2}|z_{i_1}\cdots z_{i_{l-1}}|^{2k}\,d\bar z_{i_l}.
\]
By \eqref{CIf},
\begin{equation}\label{E:exam}
T(f^k)(z)=u^k(z)-\sC_n(u^k)(z)=u^k(z)-\frac{1}{k}=\frac{1}{k}|z_I|^{2k}-\frac{1}{k}.
\end{equation}
Now define $f=\sum_kf^k$. Equation \eqref{E:exam} shows $T(f)$ does not exist, since the harmonic series diverges.

On the other hand, direct computation shows
\[
f_{I}=\sum_{k=1}^{\infty}\frac{\partial^l u^k}{\partial\bar z_{i_1}\cdots\partial\bar z_{i_l}}=z_I\sum_{k=1}^{\infty}k^{l-1}|z_I|^{2k-2}.
\]
Note that for $j=1,2,\dots,n$
\[
\int_{\D}|z_j|^{2k-1}\,dA(z_j)=2\pi\int_{0}^1r^{2k}\,dr=\frac{2\pi}{2k+1}\approx\frac{1}{k}\qquad\text{as\,\,\,}k\to\infty
\]
and
\[
\int_{\D}\,dA(z_j)=\pi\approx 1\qquad\text{as\,\,\,}k\to\infty.
\]
Therefore
\begin{align*}
\|f_{I}\|_{L^1(\D^n)}
&=\int_{\D^n}\left|z_I\sum_{k=1}^{\infty}k^{l-1}|z_I|^{2k-2}\right|\,dV(z)\\
&=\int_{\D^n}\sum_{k=1}^{\infty}k^{l-1}|z_I|^{2k-1}\,dV(z)\\
&\approx\sum_{k=1}^{\infty}k^{l-1}\left(\frac{1}{k}\right)^l=\sum_{k=1}^{\infty}\frac{1}{k}=\infty.
\end{align*}
For any $J\subset\{1,2,\dots,n\}$, $J\neq\emptyset$, and $J\neq I$, direct computation shows
\[
f_{J}=\left\{\begin{array}{lcc} 0 & \text{if} & J\not\subset I \\ z_J\sum_{k=1}^{\infty}k^{|J|-1}|z_J|^{2k-2}|z_{I\setminus J}|^{2k} & \text{if} & J\subset I \end{array}\right. .
\]
So for $J\not\subset I$, $\|f_{J}\|_{L^1(\D^n)}=0$. For $J\subset I$ and $J\neq\emptyset$, $|I|-|J|\ge 1$ and
\begin{align*}
\|f_{J}\|_{L^1(\D^n)}
&=\int_{\D^n}\left|z_J\sum_{k=1}^{\infty}k^{|J|-1}|z_J|^{2k-2}|z_{I\setminus J}|^{2k}\right|\,dV(z)\\
&=\int_{\D^n}\sum_{k=1}^{\infty}k^{|J|-1}|z_J|^{2k-1}|z_{I\setminus J}|^{2k}\,dV(z)\\
&\approx\sum_{k=1}^{\infty}k^{|J|-1}\left(\frac{1}{k}\right)^{|J|}\left(\frac{1}{k}\right)^{|I|-|J|}\le \sum_{k=1}^{\infty}O\Big(\frac{1}{k^2}\Big)<\infty.
\end{align*}
Thus assuming $\|f_{J}\|_{L^1}<\infty$ for all $J\neq I$ does not guarantee the existence of $T(f)$.

Moreover, when $|I|=l>1$, $f\in L^p_{(0,1)}(\D^n)$ for $1\le p<l$. To see this, note 
\[
f^k=f^k_{i_1}\,d\bar z_{i_1}+\cdots+f^k_{i_l}\,d\bar z_{i_l},
\]
where $f^k_{i_j}=z_{i_j}|z_{i_j}|^{2k-2}|z_I/z_{i_j}|^{2k}$ for $j=1,\dots,l$. It suffices to show $\sum_k f^k_{i_j}\in L^p(\D^n)$ for each $j$. By Minkowski's inequality, 
\begin{align*}
\left\|\sum_{k=1}^{\infty} f^k_{i_j}\right\|_{L^p(\D^n)}
&\le\sum_{k=1}^{\infty}\left\| f^k_{i_j}\right\|_{L^p(\D^n)}\\
&=\sum_{k=1}^{\infty}\left(\int_{\D^n}\left|z_{i_j}|z_{i_j}|^{2k-2}|z_I/z_{i_j}|^{2k}\right|^p\,dV(z)\right)^{1/p}\\
&\le C_p\sum_{k=1}^{\infty}\left(\frac{1}{k^l}\right)^{1/p}<\infty,
\end{align*}
provided $1\le p<l$.

In particular, if $I=\{1,2,\dots,n\}$, then $f\in L^p_{(0,1)}(\D^n)$ for $1\le p<n$ and $T(f)$ does not exist. This contrasts sharply with results on the Henkin-Cauchy-Fantappi\' e operator known on strongly pseudoconvex domains.
\end{example}

\subsection{Regularity with $L^p$ data}\label{SS:LpRegularity}
The solid Cauchy transform regularizes in $L^p$: see Lemma \ref{Linfty1dim}, note $r\le p$ there and $L^p(D)\subset L^r(D)$ since $D$ is bounded. As a result, less than $L^p$ control on the various $f_{I}$ will force an $L^p$ estimate on $T(f)$. Some preparation is needed to state the result.

First extend the superscript notation from \eqref{E:solidmultiCT} to iterated absolute Cauchy transforms. For  $I=\left\{i_1\dots , i_l\right\}\subset\{1,\dots ,n\}$, define 
\begin{equation}\label{E:solidmultiabsCT}
\left|\bm{C}\right|^{I}(u)=\frac{1}{(2\pi i)^l}\int_{D_{i_1}\times\cdots\times D_{i_l}}\frac{|u|\,d\bar\zeta_{i_1}\wedge d\zeta_{i_1}\wedge\cdots\wedge d\bar\zeta_{i_l}\wedge d\zeta_{i_l}}{\left|\zeta_{i_1}-z_{i_1}\right|\cdots\left|\zeta_{i_l}-z_{i_l}\right|}.
\end{equation}
Second, for fixed $I=\{i_1\,\dots,i_l\}\subset\{1,2,\dots,n\}$, suppose $I^c=\{j_1\,\dots,j_{n-l}\}$. The variables $z_{i_m}$ and $z_{j_k}$ are intermixed in any order. As a
notational simplification,  when $h$ is a function defined on $D$,  write $h(z', z'')$ to denote the function $h(z_1,\dots z_n)$ where $ z'$ are the variables $z_{i_1},\dots , z_{i_l}$ and $z''$ are the variables $z_{j_1}, \dots z_{j_{n-l}}$ intermixed in the order prescribed by increasing order on $I$ and $I^c$.
Next recall Minkowski's integral inequality: if $(A,\mu), (B,\nu)$ are measure spaces and $F$ is a measurable
function on $A\times B$, then
\begin{equation*}
\left(\int_A\left(\int_B\left|F(a,b)\right|\, d\nu(b)\right)^t d\mu(a)\right)^{1/t}\leq \int_B\left(\int_A\left|F(a,b)\right|^t\, d\mu(a)\right)^{1/t} d\nu(b)\quad\text{when }t\in [1,\infty).
\end{equation*}
See, e.g., \cite[Appendix A.1]{Steinbook70}. Finally, as notational shorthand let $dV(a_k) =d\bar a_k\wedge da_k$ for various symbols $a_k$.

\begin{theorem}\label{T:Lp} Let $D=D_1\times\cdots\times D_n$ be a bounded product domain in $\C^n$, where $D_k\subset\C$ are domains with piecewise $C^1$ boundaries. 
For a given $p\in [1,\infty]$, choose $r>\max\{2p/(p+2),1\}$. Let $f$ be a $\dbar$-closed $(0,1)$-form on $D$.

For each nonempty $I=\{i_1\,\dots,i_l\}\subset\{1,2,\dots,n\}$, let $I^c=\left\{j_1, \dots ,j_{n-l}\right\}$, $D'= D_{i_1}\times\dots\times D_{i_l}$, and $D''=D_{j_1}\times\dots\times D_{j_{n-l}}$. Assume 
$$\text{(H)}\qquad\qquad\left(\int_{D'} \left|f_{I}\left(z', z''\right)\right|^r\, dV\left(z_{i_1}\right)\wedge\dots\wedge dV\left(z_{i_l}\right)\right)^{1/r}\in L^p\left(D''\right).$$

Then $\dbar (Tf) =f$  and 
\begin{equation}\label{E:fancyLp}
\left\|Tf\right\|^p_{L^p(D)}\leq C \sum_{\emptyset\neq I\subset\{1,\dots,n\}}\int_{D''}\left(\int_{D'} \left|f_{I}\left(z', z''\right)\right|^r\, dV\left(z'\right)\right)^{p/r}dV(z')
\end{equation}
where $dV(z')=dV\left(z_{i_1}\right)\wedge\dots\wedge dV\left(z_{i_l}\right)$ and $dV(z'')=dV\left(z_{j_1}\right)\wedge\dots\wedge dV\left(z_{j_{n-l}}\right)$. \end{theorem}

\begin{remark}
(a) When $D''=\emptyset$, i.e. for the term in the sum corresponding to the full set $J=\{1,\dots ,n\}$, the meaning of the double integral on
the right-hand side of  \eqref{E:fancyLp} is simply $\left\|f_{J}\right\|_{L^r(D)}$ and the assumption (H) means $\left\|f_{J}\right\|_{L^r(D)}<\infty$.
\smallskip

(b) Consider the case $p=\infty$. Then $r$ is any number $>2$, hypothesis $(H)$ changes to
\begin{equation*}
\sup_{z''\in D''}\int_{D'} \left|f_{I}\left(z', z''\right)\right|^r\, dV(z') <\infty,
\end{equation*}
and conclusion \eqref{E:fancyLp} becomes
\begin{equation*}
\left\|Tf\right\|_{L^\infty (D)}\leq C \sum_{\emptyset\neq I\subset\{1,\dots,n\}}\,\sup_{z''\in D''} \left(\int_{D'} \left|f_{I}\left(z', z''\right)\right|^r dV\left(z'\right)\right)^{1/r}dV(z').
\end{equation*}
\end{remark}

\begin{proof}[Proof of Theorem \ref{T:Lp}]
Since $D$ is bounded, the hypothesis implies $f_I\in L^1(D)$ and thus $f\in \cb$. Therefore $\dbar(Tf)=f$ holds weakly by Theorem \ref{main1}. 

The  $L^p(D)$ bound on $Tf$  follows because the $L^p$ norm can be evaluated iteratively. First note
\begin{align}\label{E:Lpstep1}
\left\|T(f)\right\|_{L^p(D)}&=\left\|\sum_{\emptyset\neq I\subset\{1,\dots,n\}}\bm{C}^{I}(f_{I}^{I^c})\right\|_{L^p(D)} 
\leq \sum_{\emptyset\neq I\subset\{1,\dots,n\}}\left\|\bm{C}^{I}(f_{I}^{I^c})\right\|_{L^p(D)}
\end{align}
by the ordinary Minkowski inequality. 

Lemma \ref{Linfty1dim} and Minkowski's integral inequality can be used to show each
term in the sum \eqref{E:Lpstep1} is bounded by the right hand side of  \eqref{E:fancyLp}.
This estimation is straightforward, but tedious to notate in arbitrary dimension. Details are given for $n=2$, which contains all steps needed for the general case.

When $n=2$, $D=D_1\times D_2$ and there are three terms in the sum \eqref{E:Lpstep1} -- 
\begin{equation}\label{E:MinSum1}
\bm{C}^{\{1\}}\left(f_{\{1\}}^{\{2\}}\right) + \bm{C}^{\{2\}}\left(f_{\{2\}}^{\{1\}}\right) + \bm{C}^{\{1,2\}}\left(f_{\{1,2\}}\right). 
\end{equation}
If $f=f_1 d\bar z_1 +f_2 d\bar z_2$, recall that
$f_{\{1\}}^{\{2\}} =f_1^{\{2\}}$ and $f_{\{2\}}^{\{1\}}=f_2^{\{1\}}$, while $f_{\{1,2\}}=\frac{\partial f_1}{\partial\bar z_2}$ ($=\frac{\partial f_2}{\partial\bar z_1}$, since $\dbar f=0$). Hypothesis (H) becomes three conditions:
\begin{itemize}
\item[(i)] $\left(\int_{D_1}\left|f_1(\zeta_1, z_2)\right|^r\, dV(\zeta_1)\right)^{1/r}\in L^p\left(D_2\right),$
\medskip
\item[(ii)] $\left(\int_{D_2}\left|f_2(z_1,\zeta_2)\right|^r\, dV(\zeta_2)\right)^{1/r}\in L^p\left(D_1\right),$
\medskip
\item[(iii)] $\int_D \left|\frac{\partial f_1}{\partial\bar z_2}\right|^r\, dV(\zeta_1)\wedge dV(\zeta_2) <\infty.$
\end{itemize}

Consider the first term in \eqref{E:MinSum1},
$$\bm{C}^{\{1\}}\left(f_1^{\{2\}}\right)(z_1,z_2)=\frac 1{2\pi i}\int_{D_1}\frac {f_1(\zeta_1,z_2)\, dV(\zeta_1)}{\zeta_1-z_1}.$$
Lemma  \ref{Linfty1dim} implies
\begin{align*}\label{E:Minkow1}
\left\|\bm{C}^{\{1\}}\left(f_1^{\{2\}}\right)\right\|^p_{L^p(D_1)}\leq C \left(\int_{D_1}\left|f_1(z_1, z_2)\right|^r dV(z_1)\right)^{p/r}.
\end{align*}
Integrating both sides in $z_2$ over $D_2$ yields 
\begin{align*}
\left\|\bm{C}^{\{1\}}\left(f_1^{\{2\}}\right)\right\|_{L^p\left(D_1\times D_2\right)}^p \leq C\int_{D_2} \left(\int_{D_1}\left|f_1(z_1, z_2)\right|^r dV(z_1)\right)^{p/r} dV(z_2),
\end{align*}
which is finite by (i).
The same argument shows
\begin{equation*}
\left\|\bm{C}^{\{2\}}\left(f_2^{\{1\}}\right)\right\|_{L^p\left(D_1\times D_2\right)}^p \leq C \int_{D_1} \left(\int_{D_2}\left|f_2(z_1, z_2)\right|^r dV(z_2)\right)^{p/r} dV(z_1),
\end{equation*}
which is finite by (ii).
Thus the $L^p$ norm of the first two terms in \eqref{E:MinSum1} are bounded by the right-hand side of \eqref{E:fancyLp}.

The last term in in \eqref{E:MinSum1} involves the double Cauchy transform and is handled slightly differently. Lemma \ref{Linfty1dim} implies
\begin{align*}
\int_{D_1}\left|\bm{C}^{\{1,2\}}\left(f_{\{1,2\}}\right)(s_1,z_2)\right|^pdV(s_1)&=\int_{D_1}\left|\bm{C}^{\{1\}}\left(\bm{C}^{\{2\}}\left(f_{\{1,2\}}\right)\right)(s_1,z_2) \right|^p\, dV(s_1) \\
&\leq C\left(\int_{D_1}\left|\bm{C}^{\{2\}}\left(f_{\{1,2\}}\right)(z_1,z_2)\right|^rdV(z_1)\right)^{p/r}.
\end{align*}
Integrating both sides in $z_2$ over $D_2$ yields
\begin{equation*}
\left\|\bm{C}^{\{1,2\}}\left(f_{\{1,2\}}\right)\right\|^p_{L^p(D)}\leq C\int_{D_2}\left(\int_{D_1}\left|\bm{C}^{\{2\}}\left(f_{\{1,2\}}\right)(s_1,z_2)\right|^rdV(s_1)\right)^{p/r}dV(z_2).
\end{equation*}
Note $\frac pr \ge 1$. Raise both sides to the power $\frac rp$ and see
\begin{align*}
\left\|\bm{C}^{\{1,2\}}\left(f_{\{1,2\}}\right)\right\|^r_{L^p(D_1\times D_2)}&\lesssim\left(\int_{D_2}\left(\int_{D_1}\left|\bm{C}^{\{2\}}\left(f_{\{1,2\}}\right)(s_1,z_2)\right|^rdV(s_1)\right)^{p/r}dV(z_2)\right)^{r/p} \\
&\lesssim \int_{D_1}\left(\int_{D_2}\left|\bm{C}^{\{2\}}\left(f_{\{1,2\}}\right)(s_1,z_2)\right|^p dV(z_2)\right)^{r/p}dV(s_1)
\end{align*}
by Minkowski's integral inequality. Lemma \ref{Linfty1dim} says the last expression is
\begin{align*}
&\lesssim\int_{D_1}\left(\int_{D_2}\left|f_{\{1,2\}}(s_1,s_2)\right|^r dV(s_2)\right)dV(s_1) \\
&= \left\|f_{\{1,2\}}\right\|^r_{L^r(D_1\times D_2)}.
\end{align*}
Thus  $\left\|\bm{C}^{\{1,2\}}f_{\{1,2\}}\right\|_{L^p(D)}$ is bounded by the right-hand side of \eqref{E:fancyLp} as well.

\end{proof}

\begin{remark}\label{R:Lp}
Since $r\le p$, Theorem \ref{T:Lp} shows a ``gain'' in integrability, passing from $f$ to $Tf$. Moreover $r<<p$ as $p\to\infty$, which suggests  applications.

But Theorem \ref{T:Lp} differs from previous gain results on $\dbar$ in two respects: (i) derivatives $f_I$ appear, and (ii) the gain does not stem
from the $\dbar$-Neumann operator satisfying a subelliptic estimate.  See \cite{Krantz76} for  $L^p$ gains due to an $L^2$ subelliptic estimate.
\end{remark}

Removing the exponent $r$ from Theorem \ref{T:Lp} yields a simpler version of the basic $L^p$ estimate. The estimate illustrates how derivatives $f_{I}, |I|>1$ naturally bound $Tf$.

\begin{corollary}\label{C:Lp}
Let $p\in[1,\infty]$. Suppose $f_{I}\in L^p(D)$ for all $I\neq\emptyset$ and $\dbar f=0$. Then $\dbar(Tf)=f$ and $T$ satisfies
\[
\|T(f)\|_{L^p(D)}\le C\sum_{I\neq\emptyset}\|f_{I}\|_{L^p(D)}.
\]
\end{corollary} 

\begin{proof} For any $1\le p\le \infty$, it holds that $p>2p/(p+2)$. Choosing $r=p$ in Theorem \ref{T:Lp} yields the stated conclusion.

\end{proof}


\section{Higher dimensional factors: $L^p$ estimates}
\label{S:general}

In this section, the operator in Definition \ref{D:T} is extended to a solution operator for $\dbar$ on product domains with higher dimensional factors.

\subsection{Alternate expression in two dimensions} 
The key observation is seen by rewriting  $T$ on $D=D_1\times D_2$  with one-dimensional factors. 

If $f=f_1 d\bar z_1 +f_2 d\bar z_2\in\cb$, Theorem \ref{main1} says
\begin{align}\label{E:alt}
T(f)&=-\bm{C}^{\{1\}}\left(f_{\{1\}}^{\{2\}}\right)-\bm{C}^{\{2\}}\left(f_{\{2\}}^{\{1\}}\right)-\bm{C}^{\{1,2\}}\left(f_{\{1,2\}}\right)\notag \\
&= -\bm{C}^{\{1\}}\left(f_{1}^{\{2\}}\right)-\bm{C}^{\{2\}}\left(f_{2}^{\{1\}}\right)-\bm{C}^{\{1,2\}}\left(\frac{\partial f_1}{\partial\bar z_2}\right), 
\end{align}
solves $\dbar(Tf)=f$.
The second equality unravels the subscripts. Recall the superscripts in $f_1^{\{2\}}, f_2^{\{1\}}$ 
indicate evaluation, e.g., if $f_1\in C(\overline D)$
$$\bm{C}^{\{1\}}\left(f_1^{\{2\}}\right)(z_1,z_2)=\frac 1{2\pi i}\int_{D_1}\frac {f_1(\zeta_1,z_2)\, d\bar\zeta_1\wedge d\zeta_1}{\zeta_1-z_1}.$$

 Let $\dbar_j$ be the $\dbar$-operator in the variable $z_j$: $\dbar_j u = \frac{\partial u}{\partial\bar z_j} d\bar z_j$
for $u=u(z_1,z_2)\in C^1(D)$. Thus
$\dbar u= \dbar_1 u +\dbar_2 u$. If $S$ is an operator acting on functions, define $S$ on a $(0,1)$-form, such as $f$, by $S(f)= S(f_1) +S(f_2)$.

Suppose $g\in C(\overline{D})$. Theorem \ref{integralformula} with $n=1$ says
\[
S_j(g)=-\bm{C}^{\{j\}}(g)=\frac{-1}{2\pi i}\int_{D_j}\frac{g\,d\bar\zeta_j\wedge d\zeta_j}{\zeta_j-z_j}
\]
solves the $\dbar$-equation $\dbar_j (S_jg)= g d\bar z_j$ on $D_j$ for $j=1,2$. Thus \eqref{E:alt}  can be written
\begin{align}\label{E:alt2}
T(f)
&=S_1\left(f^{\{2\}}_1\right)+S_2\left(f^{\{1\}}_2\right)+S_2\left(-S_1\left(f_{\{1,2\}}\right)\right)\notag\\
&=S_1\left(f^{\{2\}}_1\right)+S_2\left(f^{\{1\}}_2-S_1\left(\dbar_2\left(f_1^{\{2\}}\right)\right)\right)\notag\\
&=S_1\left(f^{\{2\}}_1\right)+S_2\left(f-\dbar S_1\left(f_1^{\{2\}}\right)\right).
\end{align}
The last equality -- which is the crucial observation -- holds since the operators $S_1$ and $\frac{\partial}{\partial\bar z_2}$ commute and $\dbar_1S_1(f_1)=f_1d\bar z_1$, yielding the cancellation inside the parentheses. 
Thus $Tf$ is written as the sum of solutions to $\dbar$ problems on the factors $D_1, D_2$. On the other hand, the data for these $\dbar$ problems involves more than restricting $f$ to the separate factors.

This turns out to hold in greater generality -- in particular when solution operators on the factors are not given by integrals.

\subsection{Inductive argument on factors}

The argument giving the last equality in  \eqref{E:alt2} is generalized to arbitrary products.

\subsubsection{Notation}
Let $D\subset \C^N$ be a product domain of the form $D=D_1\times\cdots\times D_k$, each $D_j\subset\C^{n_j}$ a bounded domain, $n_1+\cdots+n_k=N$. The coordinates on $D$ will be written in several
ways, depending on context. Define
\begin{align*}
\left(z_1,\dots z_N\right)=\left(\bm{z}^1,\dots ,\bm{z}^k\right) =\left(z_1^1,\dots , z^1_{n_1}, z^2_1,\dots , z^2_{n_2},\dots , z^k_1,\dots , z^k_{n_k}\right)
\end{align*}
where $\bm{z}^j=\left(z^j_1,\dots ,z^j_{n_j}\right)$ are the standard coordinates on $D_j$. The $\dbar$ operator on $D$ can be decomposed into sub-$\dbar$ operators; for $u\in C^1(D)$, define
\begin{align*}
\dbar u =\sum_{j=1}^{n_1}\frac{\partial u}{\partial \bar z^1_j}d\bar z^1_j+\dots + \sum_{j=1}^{n_k} \frac{\partial u}{\partial\bar z^k_j}d\bar z^k_j =:\dbar _1 u+\dots +\dbar_k u.
\end{align*}
If $f$ is a $(0,1)$-form on $D$, its components can be rearranged to define
\begin{align}\label{D:form_decomp}
f=\sum_{j=1}^{k}\left(\sum_{i=1}^{n_j}f_{i}^j\,d\bar z_{i}^j\right) =: \sum_{j=1}^k \bm{f}^j.
\end{align}
Note each $\bm{f}^j$ contains only the differentials $d\bar z^j_1,\dots ,d\bar z^j_{n_j}$ and is a well-defined $(0,1)$-form on $D_j$. However the components of $\bm{f}^j$ are functions of
the full set of variables $(z_1,\dots ,z_N)$, not only the variables $\left( z^j_1,\dots ,z^j_{n_j}\right)$. If $f$ is written as \eqref{D:form_decomp}, define a projection operator $\pi_j:L^p_{(0,1)}(D)\to L^p_{(0,1)}(D_j)$ by
$\pi_j\left(f\right)=\bm{f}^j$
for $j=1,\dots,k$.

It follows immediately that $\dbar f=0$ on $D$ implies $\dbar_j\left(\pi_j\left(f\right)\right)=0$ for any  $j=1\dots , k$. Notice however that the system of equations (*) $\dbar_j\left(\pi_j\left(f\right)\right)=0$ for all $j=1\dots , k$
does not imply that $\dbar f=0$. 
In particular, system (*) gives no information on the various $\dbar_i\left(\pi_j\left(f\right)\right), i\neq j$. 

It is useful to have notation for a derivative with respect to a variable comprising a vector $\bm{z}^j$ without specifying the individual variable. 
 For $u$ a function defined on $D$, let
$$\frac{\partial u}{\partial \bm{z}^j_*}=\frac{\partial u}{\partial z^j_m}\qquad \text{for a  single, unspecified}\quad m=1, \dots , n_j.$$
Thus $\frac{\partial u}{\partial \bm{z}^j_*}$ represents a class of derivatives: all singletons from the full set $\left\{\frac{\partial u}{\partial z^j_1}, \dots , \frac{\partial u}{\partial z^j_{n_j}}\right\}$ of first order partials with respect to the coordinates $\left(z^j_1,\dots ,z^j_{n_j}\right)$ that comprise $\bm{z}^j$. Similarly define $\frac{\partial u}{\partial\bm{\bar z}^j_*}$. When $u\notin C^1(D)$, these derivatives
are interpreted in the sense of distributions. Similar notation is used on $(0,1)$-forms; if $f$ is given by \eqref{D:form_decomp},
$$\left(\pi_j\left(f\right)\right)_* = f^j_m \qquad \text{for a single, unspecified}\quad m=1, \dots , n_j.$$

The natural extension  of subscripts \eqref{E:subl_form1} to higher dimensional factors  is relatively easy to express using the ${}_*$-notation. 

\begin{definition}
\label{fI}
Let $D\subset \C^N$ be a product domain of the form $D=D_1\times\cdots\times D_k$, where each $D_j\subset\C^{n_j}$ is a bounded domain. Let $f$ be a $(0,1)$-form on $D$ expressed as \eqref{D:form_decomp}.

For any $I=\{i_1,\dots,i_l\}\subset\{1,\dots ,k\}$ with $l=|I|\ge1$, define
\begin{equation}\label{E:subl_form2}
f_{I}=\frac{\partial^{l-1}\left(\pi_{i_1}\left(f\right)\right)_*}{\partial\bm{\bar z}_*^{i_2}\cdots\partial\bm{\bar z}_*^{i_l}}.
\end{equation}
\end{definition}

\begin{remark} This extends  \eqref{E:subl_form1} since, when each $D_j\subset \C$, only one choice occurs for each $\partial\bm{\bar z}_*^{i_m}$ and $\left(\pi_{i_1}\left(f\right)\right)_*=f_{i_1}$.
\end{remark}

When the components of $f$ are $\notin C^{k-1}(D)$, the derivatives $f_{I}$ in \eqref{E:subl_form2} are interpreted weakly, as before.

Conditions on $f_{I}$, such as $f_{I}\in X$ for some normed space $X$, will mean that {\it every} derivative in the form given by Definition \ref{fI} -- for all choices implicit in the subscripts $(\cdot)_*$ in the
numerator and denominator of \eqref{E:subl_form2} -- satisfies the condition. The norm $\|f_I\|_{X}$ is the sum of the norms of {\it all} derivatives in the form of \eqref{E:subl_form2}.

\subsubsection{Algebraic lemmas}

\begin{lemma}
\label{dbarcomp}
Let $f$ be a weakly $\dbar$-closed $(0,1)$-form on $D=D_1\times\cdots\times D_k$. If $f_I$ exist for all the indices $I\neq0$ in the weak sense, the $f_I$ is independent of the order of $\{i_1,\dots,i_l\}\subset\{1,2,\dots,k\}$.

Moreover for such $I=\{i_1,\dots,i_l\}$, the $(0,1)$-forms in the class
\[
\frac{\partial^{l-1} \pi_{i_1}(f)}{\partial\bm{\bar z}_{*}^{i_2}\cdots\partial\bm{\bar z}_{*}^{i_l}}
\]
are all $\dbar_{i_1}$-closed on $D_{i_1}$ in the weak sense. 
\end{lemma}

\begin{proof}
The first claim follows directly from $\dbar f=0$. The second claim follows from the $\dbar$-closedness of $f$ and the fact that $\dbar_{i_1}$ and $\partial/\partial \bar z_j$ commute. 
\end{proof}

\begin{lemma}
\label{crossdbar}
Let $f$ be a weakly $\dbar$-closed $(0,1)$-form on $D=D_1\times\cdots\times D_k$. If $\pi_j(f)=0$ for some $j\in\{1,\dots,k\}$, then
\[
\frac{\partial(\pi_i(f))_*}{\partial\bm{\bar z}^{j}_*}=0\qquad\text{for any }i\neq j.
\]
\end{lemma}

\begin{proof}
The condition $\pi_j(f)=0$ implies that every component $f^j_m$ in the class $\left(\pi_j\left(f\right)\right)_*$ is identically 0. A fortiori 
\begin{equation}\label{E:helper}
\frac{\partial f^j_m}{\partial\bar z_{\ell}}=0\qquad\text{for all }\ell\in\{1,\dots,N\} \,\text{and } m\in\{1,\dots ,n_j\}.
\end{equation}
Let $i\neq j$ and consider a particular component of $\pi_i(f)_*$, say $f^i_n$ for $n\in\{1, \dots , n_i\}$. Since $\dbar f=0$, for any $\mu\in\{1,\dots ,n_j\}$
$$\frac{\partial f^i_n}{\partial\bar z^j_\mu} =\frac{\partial f^j_\mu}{\partial\bar z_\ell},$$
for some $\ell$. The conclusion follows from \eqref{E:helper}.
\end{proof}

\subsubsection{Existence and basic estimate} Suppose there are operators $T_j$ satisfying $\dbar_j\left(T_j\alpha_j\right)=\alpha_j$ when $\dbar_j\alpha_j=0$ on $D_j$, for each of the
factors in $D=D_1\times\cdots\times D_k$. Commutators of $T_j$ and barred derivatives {\it not} in directions $\bm{\bar z}^j$ -- i.e. derivatives in directions corresponding to $D_1,\dots,D_{j-1},D_{j+1},\dots,D_k$ -- arise in proving existence and regularity of a solution operator for $\dbar$ on $D$. 

When these commutators vanish, existence of the solution operator can be shown.  For clarity, an estimate needed in the proof is relegated to Proposition \ref{LpforhighT} below. 

\begin{proposition}
\label{highdimexistence}
Let $D=D_1\times\cdots\times D_k$, with $D_j\subset\C^{n_j}$, and $p\in[1,\infty]$.  

\begin{itemize}
\item[(C)] Assume on each factor there is a linear bounded operator $T_j:L^p_{(0,1)}(D_j)\to L^p(D_j)$ that solves the $\dbar_j$-equation on $D_j$ and commutes with  all the barred derivatives on $D_1,\dots,D_{j-1},D_{j+1},\dots,D_k$.
Generically denote these derivatives  $\partial/\partial\bm{\bar z}_*$.
\end{itemize}

Then there is a linear operator $T$ on
\[
\sB^p:=\{f\,\dbar\text{-closed}\,|\,f_I\in L^p(D)\text{ for all }I\neq0\}
\]
satisfying $\dbar T(f)=f$.
\end{proposition}

\begin{proof}
The proof proceeds by recursively updating the $L^p$ (0,1)-form data and solving $\dbar$-equations on the factors. The proof is slightly subtle.

Let $f\in\sB^p$ be fixed. As starting $\dbar$-data, take $\pi_1(f)$; note $\dbar_1\big(\pi_1(f)\big)=0$ by Lemma \ref{dbarcomp}. The initial $\dbar$-problem is $\dbar_1 u=\pi_1(f)$ on $D_1$. To
facilitate writing the recursion relations, define
\begin{equation*}
g_1=f\quad\text{and}\quad v^1=T_1\big(\pi_1(g_1)\big).
\end{equation*}
In $\pi_1(g_1)$, the variables $\left(\bm{\bar z}^2,\dots ,\bm{\bar z}^k\right)$ are fixed (or viewed as parameters); the operator $T_1$ acts only on the $\bm{\bar z}^1$ variables. By hypothesis, $\dbar_1 v^1=\pi_1(g_1)$. However $v^1$ also depends on the variables $\left(\bm{\bar z}^2,\dots ,\bm{\bar z}^k\right)$; subsequent $\dbar$-problems must account for extra terms created by $\dbar_\ell v^1$, $\ell\neq 1$. Thus define
$g_2=g_1-\dbar v^1$ and $v^2=T_2\big(\pi_2(g_2)\big)$.
Note the use of the full $\dbar$ on $v^1$. As before $\dbar_2\big(\pi_2(g_2)\big)=0$ and $\dbar_2 v^2=\pi_2(g_2)$ on $D_2$. A general recursion is now evident: for $j=3,\dots k$, define
\begin{equation}\label{E:recursion}
g_j=g_{j-1}-\dbar v^{j-1}\quad\text{and}\quad v^j=T_j\big(\pi_j(g_j)\big).
\end{equation}
Commutativity conditions (C) show that each $g_j$ belongs to $L^p_{(0,1)}(D)$; details of this are given in Proposition \ref{LpforhighT}. It follows that $\pi_j\left(g_j\right)\in L^p_{(0,1)}\left(D_j\right)$ and $T_j\big(\pi_j\left(g_j\right)\big)$ is well-defined.

Another consequence of (C) is needed to continue.

\begin{lemma}\label{L:vanishing} Assume the hypotheses of Proposition \ref{highdimexistence}.  For $j=2,\dots, k$
\begin{itemize}
\item[(a)] $\pi_1\left(g_j\right)=\dots =\pi_{j-1}\left(g_j\right)=0$.
\smallskip
\item[(b)] $\pi_1\left(\dbar v^j\right)=\dots =\pi_{j-1}\left(\dbar v^j\right)=0$.
\smallskip
\item[(c)] $\pi_k\left(g_k\right)= g_k$.
\end{itemize}
\end{lemma}

\begin{proof}[Proof of Lemma]
(a), (b), and (c) are proved together, by induction. For $j=2$, 
\begin{align*}
\pi_1(g_2)&=\pi_1(g_1)-\pi_1\left(\dbar v^1\right)=\pi_1(g_1)-\dbar_1 v^1\\
&= \pi_1(g_1)-\dbar_1\left(T_1\left(\pi_1\left(g_1\right)\right)\right)=0.
\end{align*}
For $2\le j<k$, assume $\pi_1(g_j),\dots,\pi_{j-1}(g_j)$ are all $0$. We claim this also holds for $j+1$. Let $\partial/\partial\bm{\bar z}_*$ be one of the barred derivatives on $D_1,\dots,D_{j-1}$. By  commutativity,
\[
\frac{\partial}{\partial\bm{\bar z}_*}v^j=\frac{\partial}{\partial\bm{\bar z}_*}T_j(\pi_j(g_j))=T_j(\frac{\partial}{\partial\bm{\bar z}_*}\pi_j(g_j)).
\]
Since $\pi_1(g_j),\dots,\pi_{j-1}(g_j)$ are all $0$ and $g_j$ is $\dbar$-closed, Lemma \ref{crossdbar} says the right hand side of the above equation is $0$. 
In particular this implies $\pi_1(\dbar v^j),\dots,\pi_{j-1}(\dbar v^j)$ are all $0$ as well. From the recursion $g_{j+1}=g_j-\dbar v^j$, it follows that $\pi_1(g_{j+1}) =\dots =\pi_{j-1}(g_{j+1})=0$. Finally $\pi_j(g_{j+1})=\pi_j(g_j)-\dbar_jv^j=0$,
since $v^j=T_j\left(\pi_j\left(g_j\right)\right)$. This proves the claim. Thus (a) holds.

However once (a) holds, $\pi_1\left(\dbar v^j\right)=\dots =\pi_{j-1}\left(\dbar v^j\right)=0$ necessarily follows, since it was an intermediate conclusion in the previous induction argument. Thus (b) holds. 
Finally, (c) holds since (a) implies $g_k=\sum_{j=1}^k\pi_j\left(g_k\right)=\pi_k(g_k)$. 
\end{proof}

To conclude the construction, let $T(f)=v^1+\cdots+v^k$. Clearly $T$ is linear on $\sB^p$. To verify $\dbar T(f)=f$,  compute
\begin{align*}
\dbar T(f)
&=\dbar(v^1+\cdots+v^k)\\
&=(g_1-g_2)+\cdots+(g_{k-1}-g_k)+(\dbar_1+\cdots+\dbar_k)v^k\\
&=g_1-g_k+\dbar_k(v^k)\\
&=f-g_k+\pi_k(g_k)=f.
\end{align*}
The second equality follows from the recursion \eqref{E:recursion}, the third and fourth equalities follow from Lemma \ref{L:vanishing}.
\end{proof}

\begin{remark}
\label{1dimcompresult}
If each $D_j$ is one-dimensional with piecewise $C^1$ boundary and $T_j=-\bm{C}$, $\bm{C}$ defined in \eqref{E:solidCT}, then $u=T(f)$ for the $T$ given in Proposition  \ref{highdimexistence} is exactly the solution in Theorem \ref{main1}.
\end{remark}

The $L^p$ estimate needed in the proof of Proposition  \ref{highdimexistence}, which also establishes $L^p$-regularity of $T$, is now proved.

\begin{proposition}\label{LpforhighT}
Let $D=D_1\times\cdots\times D_k$, with $D_j\subset\C^{n_j}$, and $p\in[1,\infty]$.  

Assume on each factor there is a linear bounded operator $T_j:L^p_{(0,1)}(D_j)\to L^p(D_j)$ that solves the $\dbar_j$-equation on $D_j$ and commutes with  all the barred derivatives in the variables on $D_1,\dots,D_{j-1},D_{j+1},\dots,D_k$. Generically denote these derivatives  $\partial/\partial\bm{\bar z}_*$.

Then for each $g_j$ defined in Proposition \ref{highdimexistence}, $g_j\in L^p_{(0,1)}(D)$. Furthermore,
there is a constant $C>0$ independent of $f\in\sB^p$ such that
\[
\|T(f)\|_{L^p(D)}\le C\sum_{|I|\neq0}\|f_I\|_{L^p(D)}.
\]
\end{proposition}

\begin{proof}
Let $f\in\sB^p$ and fix $\left(\bm{z}^2,\dots,\bm{z}^k\right)\in D_2\times\cdots\times D_k$. As in the proof of  Proposition  \ref{highdimexistence}, consider the $\dbar_1$-equation on $D_1$
\begin{equation}\label{v1}
\dbar_1v^1=\pi_1(g_1)=f^1_1\,d\bar z^1_1+\cdots+f^1_{n_1}\,d\bar z^1_{n_1}.
\end{equation}
Let $\partial/\partial\bm{\bar z}_*$ be an arbitrary barred derivative in the variables on $D_2,\dots,D_k$. Consider another $\dbar_1$-equation
\begin{equation}
\label{dv1}
\dbar_1w^1=\frac{\partial}{\partial\bm{\bar z}_*}\pi_1(g_1)=\frac{\partial f^1_1}{\partial\bm{\bar z}_*}\,d\bar z^1_1+\cdots+\frac{\partial f^1_{n_1}}{\partial \bm{\bar z}_*}\,d\bar z^1_{n_1},
\end{equation}
viewed as paired with \eqref{v1}.

By Lemma \ref{dbarcomp}, the right hand sides of \eqref{v1} and \eqref{dv1} are well-defined and $\dbar_1$-closed. By assumption, it holds that $v^1=T_1(\pi_1(g_1))$ solves \eqref{v1} and satisfies
\begin{equation}
\label{v^1Lp}
\|v^1\|_{L^p(D_1)}\le C\|\pi_1(g_1)\|_{L^p_{(0,1)}(D_1)}.
\end{equation}
On the other hand, $w^1=T_1\left(\partial/\partial\bm{\bar z}_*(\pi_1(g_1))\right)$ solves  \eqref{dv1}.
By commutativity, 
$$w^1=T_1(\partial/\partial\bm{\bar z}_*(\pi_1(g_1)))=\partial/\partial\bm{\bar z}_*(T_1(\pi_1(g_1)))=\partial/\partial\bm{\bar z}_*(v^1).$$ Therefore
\[
\left\|\frac{\partial v^1}{\partial \bm{\bar z}_*}\right\|_{L^p(D_1)}\le C\left\|\frac{\partial \pi_1(g_1)}{\partial \bm{\bar z}_*}\right\|_{L^p_{(0,1)}(D_1)}.
\]
Recall that $g_1=f\in\sB^p$. Taking $p$th powers and integrating over $D_2\times\dots\times D_k$, the above implies $g_2=g_1-\dbar v^1\in L^p_{(0,1)}(D)$. Moreover,
\begin{equation}
\label{ftildeB1}
\begin{split}
\left\|\pi_2(g_2)\right\|_{L^p_{(0,1)}(D_1)}
&\le\|\pi_2(g_1)\|_{L^p_{(0,1)}(D_1)}+\left\|\dbar_2v^1\right\|_{L^p_{(0,1)}(D_1)}\\
&\le C\left(\|\pi_2(g_1)\|_{L^p_{(0,1)}(D_1)}+\sum_{j=1}^{n_2}\left\|\frac{\partial \pi_1(g_1)}{\partial \bar z_j^2}\right\|_{L^p_{(0,1)}(D_1)}\right).
\end{split}
\end{equation}

Next, as in the proof of  Proposition  \ref{highdimexistence}, consider the $\dbar_2$ problem on $D_2$ 
\[
\dbar_2 v^2=\pi_2(g_2).
\]
By assumption $v^2=T_2(\pi_2(g_2))$ solves this equation and satisfies
\begin{equation}
\label{v^2Lp}
\left\|v^2\right\|_{L^p(D_2)}\le C\|\pi_2(g_2)\|_{L^p_{(0,1)}(D_2)}.
\end{equation}
Just as for $v^1$, it follows that
$$\left\|\frac{\partial v^2}{\partial \bm{\bar z}_*}\right\|_{L^p(D_2)}\le C\left\|\frac{\partial \pi_2(g_2)}{\partial \bm{\bar z}_*}\right\|_{L^p_{(0,1)}(D_2)},$$
where $\partial/\partial\bm{\bar z}_*$ now denotes an arbitrary barred derivative in the variables on $D_3, \dots,D_k$. This implies $g_3=g_2-\dbar v^2\in L^p_{(0,1)}(D)$.
Combining \eqref{v^1Lp}--\eqref{v^2Lp} and integrating over $D$ yields
\begin{align*}
\|v^1+v^2\|_{L^p(D)}
&\le C\left(\|\pi_1(g_1)\|_{L^p_{(0,1)}(D)}+\|\pi_2(g_2)\|_{L^p_{(0,1)}(D)}\right)\\
&\le C\left(\|\pi_1(f)\|_{L^p_{(0,1)}(D)}+\|\pi_2(f)\|_{L^p_{(0,1)}(D)}+\sum_{j=1}^{n_2}\left\|\frac{\partial \pi_1(f)}{\partial \bar z_j^2}\right\|_{L^p_{(0,1)}(D)}\right)\\
&\le C\left(\|f_{\{1\}}\|_{L^p(D)}+\|f_{\{2\}}\|_{L^p(D)}+\|f_{\{1,2\}}\|_{L^p(D)}\right).
\end{align*}

The argument can be continued through the factors. This gives $g_j\in L^p_{(0,1)}(D)$ for all $j$ and $T(f):=v^1+\cdots+v^k$ satisfies
\[
\|T(f)\|_{L^p(D)}\le C\sum_{|I|\neq0}\|f_I\|_{L^p(D)}.
\]
\end{proof}

\begin{remark}\label{R:explainC}
The fact that $T(f)$ solves $\dbar$ on $D$ only requires that $T_j$ commutes with  all $\partial/\partial\bm{\bar z}_*$  on $D_1,\dots,D_{j-1}$ (cf. the proof of Proposition \ref{highdimexistence}).  
Guaranteeing $g_j$ belongs to the specified $L^p_{(0,1)}(D)$ also requires that $T_j$ commutes with $\partial/\partial\bar z_*$ on $D_{j+1},\dots,D_k$ (cf. the proof of Proposition \ref{LpforhighT}). 

However, if a {\it specified} $p$ is changed to {\it some} $p\in(1,\infty)$, the requirement that $T_j$ commutes with $\partial/\partial\bm{\bar z}_*$ on $D_{j+1},\dots,D_k$ is redundant. One can show $\partial/\partial\bm{\bar z}_*(T_j(\pi_j(g_j)))$ is in $L^p(D)$ (hence so is $g_{j+1}$) using difference quotients; see Lemma \ref{differentialquotient} below.
\end{remark}

\subsubsection{Conclusion}

The following summarizes results of the previous section.

\begin{theorem}
\label{main2}
Let $p\in[1,\infty]$, $D=D_1\times\cdots\times D_k$, where each factor $D_j\subset\C^{n_j}$ is a bounded domain,  and $n_1+\cdots+n_k=N$. Consider the $\dbar$-equation on $D$
\begin{equation}
\label{dbaronD}
\dbar u=f,
\end{equation}
where $f$ is a weakly $\dbar$-closed $(0,1)$-form.

Assume for each factor $D_j$, the $\dbar_{j}$-equation is solvable by a linear bounded operator $T_j:L^p_{(0,1)}(D_j)\to L^p(D_j)$, that commutes with  all the barred derivatives on $D_1,\dots,D_{j-1},$ $D_{j+1},\dots,D_k$.

If $f_I\in L^p(D)$ for all $I\neq0$, then there exists a solution $u=T(f)$ of the equation \eqref{dbaronD} with the estimate
\begin{equation}
\label{Bestimate}
\|u\|_{L^p(D)}\le C\sum_{|I|\neq0}\|f_I\|_{L^p(D)}.
\end{equation}
The operator $T$ is linear and bounded from $\sB^p$ to $L^p(D)$.
\end{theorem}

The commutativity assumption in Theorem \ref{main2} is rather mild. As a step towards seeing this, note 
\begin{lemma}\label{L:IntegralsCommute}
\label{integralcommute}
If each $T_j$ is an integral solution operator for $\dbar_j$ on $D_j$, then $T_j$ commutes with directional derivatives on the other factors.
\end{lemma}

\begin{proof}
This is a direct consequence of pairing weak derivatives with a test function and applying Fubini's theorem.
\end{proof}

Let $W^{k,p}(D)$ denote the usual $L^p$ Sobolev space of derivative order $k$: the measurable functions $f$
such that 
\begin{equation*}\label{D:LpSobolevNorm}
\norm{f}_{W^{k,p}(D)} = \left( \sum_{|\alpha|\le k} \int_{D} \left|\partial^\alpha f\right|^p\,dV \right)^{\frac1p}
\end{equation*} 
is finite, where derivatives are interpreted weakly. $W^{k,p}_{(0,1)}(D)$ denotes the $(0,1)$-forms with components in $W^{k,p}(D)$.

\begin{corollary}\label{C:Wkp}
Let $D=D_1\times\cdots\times D_k\subset\C^N$, with $D_j\subset\C^{n_j}$ bounded. Assume that each $D_j$ is strongly pseudoconvex with $C^2$ boundary. 

For $1\le p\le\infty$, there is a solution $u=T(f)$ of the equation \eqref{dbaronD} with $L^p$ estimate
\[
\|u\|_{L^p(D)}\le C_p\sum_{|I|\neq0}\|f_I\|_{L^p(D)},
\]
if the right hand side is finite. In particular, if $f\in W_{(0,1)}^{k-1,p}(D)$, then
\[
\|u\|_{L^p(D)}\le C_p\|f\|_{W_{(0,1)}^{k-1,p}(D)}.
\]
\end{corollary}

\begin{proof}
Theorem 6.1 and Proposition 6.4 in \cite{Ovrelid} give existence and regularity of a solution operator on each factor  $D_j$ of $D$. Since the solution operator on each $D_j$ is an integral operator, Lemma \ref{integralcommute} guarantees the commutativity 
needed to apply Theorem \ref{main2}.
\end{proof}


\section{Higher dimensional factors: $L^2$-Sobolev estimates}
\label{smoothregularity}

\subsection{Commutative lemmas}
The following lemma guarantees  functions constructed later are in the right space.

\begin{lemma}
\label{differentialquotient}
Let $D_1\subset\C^{n_1}$ and $D_2\subset\C^{n_2}$ be bounded domains, where $n_1$ and $n_2$ are positive integers. Assume $p\in(1,\infty)$. Let $T:L^p(D_1)\to L^p(D_1)$ be a bounded linear operator. Let $\partial_{\nu}$ be a directional derivative along the unit vector $\nu$ in $D_2$. If $g\in L^p(D_1\times D_2)$ and $\partial_{\nu}g\in L^p(D_1\times D_2)$, then $\partial_{\nu}T(g)\in L^p(D_1\times D_2)$ and
\[
\int_{D_1\times D_2}|\partial_{\nu}T(g)(z)|^p\,dV(z)\le C\int_{D_1\times D_2}|\partial_{\nu}g(z)|^p\,dV(z)
\]
for some $C>0$.
\end{lemma}

\begin{proof}
Use notation $z=(z^1,z^2)\in D_1\times D_2$. For any function $h$ on $D_2$, define the difference quotient along the direction $\nu$ of size $\delta\neq0$ by
\[
\Delta^{\delta}_{\nu}h(z^2)=\frac{h(z^2+\delta\nu)-h(z^2)}{\delta},
\]
where $z^2,z^2+\delta\nu\in D_2$.

Since $g,\partial_{\nu}g\in L^p(D_1\times D_2)$, Fubini's theorem implies $g(z^1,\cdot),\partial_{\nu}g(z^1,\cdot)\in L^p(D_2)$ for a.e. $z^1\in D_1$. The smooth approximation arguments in 
\cite[Lemma 7.2, Lemma 7.3 and Theorem 7.9]{GilbargTrudinger} also work for directional derivatives (or non-isotropic Sobolev spaces---indeed, only one derivative in $L^p$ is considered). So \cite[Lemma 7.23]{GilbargTrudinger} also holds for $\partial_{\nu}g(z^1,\cdot)$, i.e. for a.e. $z^1\in D_1$
\begin{equation}
\label{D'2byD2}
\int_{D'_2}|\Delta^{\delta}_{\nu}g(z^1,z^2)|^p\,dV(z^2)\le\int_{D_2}|\partial_{\nu}g(z^1,z^2)|^p\,dV(z^2)
\end{equation}
for any $D'_2\subset\subset D_2$ satisfying $0<|\delta|<\text{dist}(D'_2,bD_2)$.

Since $g\in L^p(D_1\times D_2)$, $T(g)\in L^p(D_1\times D_2)$. For any $D'\subset\subset D_1\times D_2$ satisfying $0<|\delta|<\text{dist}(D',b(D_1\times D_2))$, if it holds that
\begin{equation}
\label{D'byD1D2}
\int_{D'}|\Delta^{\delta}_{\nu}T(g)|^p\,dV(z)\le C\int_{D_1\times D_2}|\partial_{\nu}g(z)|^p\,dV(z)
\end{equation}
for a constant $C>0$, then \cite[Lemma 7.24]{GilbargTrudinger} implies the conclusion.

To verify \eqref{D'byD1D2}, define $(D')_2=\{z^2\in\C^{n_2}\,|\,z=(z^1,z^2)\in D'\}$. Since $(D')_2$ is the image of a coordinate projection from $D'$ and such a projection is an open map, $(D')_2$ is a bounded domain. Moreover, $D'\subset D_1\times(D')_2$, $(D')_2\subset\subset D_2$, and $0<|\delta|<\text{dist}(D',b(D_1\times D_2))\le\text{dist}((D')_2,bD_2)$. So
\begin{align*}
\int_{D'}|\Delta^{\delta}_{\nu}T(g)|^p\,dV(z)
&\le\int_{(D')_2}\int_{D_1}|\Delta^{\delta}_{\nu}T(g)(z^1,z^2)|^p\,dV(z^1)\,dV(z^2)\\
&=\int_{(D')_2}\int_{D_1}|T(\Delta^{\delta}_{\nu}g)(z^1,z^2)|^p\,dV(z^1)\,dV(z^2)\\
&\le C\int_{(D')_2}\int_{D_1}|\Delta^{\delta}_{\nu}g(z^1,z^2)|^p\,dV(z^1)\,dV(z^2)\\
&\le C\int_{D_1\times D_2}|\partial_{\nu}g(z)|^p\,dV(z).
\end{align*}
The first line follows from $D'\subset D_1\times(D')_2$ and Fubini's theorem; the second line follows from the linearity of $T$; the third line follows from the boundedness of $T$; and the last line follows from Fubini's theorem and \eqref{D'2byD2}.
\end{proof}

\begin{remark}
Lemma \ref{differentialquotient} says that even without commutativity assumptions on $T$, the norm of $\partial_{\nu}T(g)$ is still controlled by the norm of $\partial_{\nu}g$.
\end{remark}

Lemma \ref{differentialquotient} can be used to show commutativity between special solution operators for $\dbar_j$ on $D_j$ and differential operators along directions in other factors.

\begin{lemma}
\label{commutativity}
Let $p$, $D_1$, $D_2$, and $\partial_{\nu}$ be as in Lemma \ref{differentialquotient}. Let $T:L_{(0,1)}^p(D_1)\to L^p(D_1)$ be a bounded linear operator which solves the $\dbar_1$-equation on $D_1$. 

Assume there exists a projection operator $P:L^p(D_1)\to A^p(D_1):=L^p(D_1)\cap\co(D_1)$, which preserves $A^p(D_1)$. 

Then the bounded operator $S:=T-P\circ T:L_{(0,1)}^p(D_1)\to L^p(D_1)$ also solves the $\dbar_1$-equation on $D_1$. Moreover, if $g\in L^p_{(0,1)}(D_1\times D_2)$, $\partial_{\nu}g\in L^p_{(0,1)}(D_1\times D_2)$, and $\dbar_1(g)=0$, then $[\partial_{\nu},S](g)=0$, where $[\partial_{\nu},S]$ denotes the commutator of $\partial_{\nu}$ and $S$.
\end{lemma} 

\begin{proof}
 $S$ is $L^p$-bounded since $P$ and $T$ are. Since  the range of $P$ is contained in $\co(D_1)$, $S$ solves the $\dbar_1$-equation on $D_1$ as well.

Note the conclusion of Lemma \ref{differentialquotient} holds with functions replaced by $(0,1)$-forms in the hypothesis. So $\partial_{\nu}S(g)\in L^p(D_1\times D_2)$. Since $\partial_{\nu}g\in L^p_{(0,1)}(D_1\times D_2)$, $S(\partial_{\nu}g)\in L^p(D_1\times D_2)$ as well. Therefore $[\partial_{\nu},S](g)\in L^p(D_1\times D_2)$.

Lemma \ref{dbarcomp} applies to $\partial_{\nu}g$, so $\partial_{\nu}g$ is $\dbar_1$-closed. Thus
\[
\dbar_1\circ[\partial_{\nu},S](g)=\dbar_1\left(\partial_{\nu}S(g)-S(\partial_{\nu}g)\right)=\partial_{\nu}\dbar_1S(g)-\partial_{\nu}g=0,
\]
which implies that $[\partial_{\nu},S](g)$ is holomorphic in $z^1$ on $D_1$. Hence $[\partial_{\nu},S](g)(\cdot,z^2)\in A^p(D_1)$ for a.e. $z^2\in D_2$.

On the other hand, $P(h)$ is holomorphic on $D_1$ for each $h\in L^p(D_1)$. Given $z^1\in D_1$, the mean-value property of holomorphic functions and Holder's inequality imply
\[
|P(h)(z^1)|=\left|\frac{1}{V(B)}\int_{B}P(h)(\zeta^1)\,dV(\zeta^1)\right|\le C_{p,z^1}\|P(h)\|_{L^p(D_1)}\le C_{p,z^1}\|h\|_{L^p(D_1)},
\]
where $B\subset D_1$ is a ball centered at $z^1$. This says the linear functional $l_{z^1}(h):=P(h)(z^1)$ is bounded on $L^p(D_1)$. By duality of $L^p(D_1)$, $P$ can be represented as an integral operator. Note that Lemma \ref{integralcommute} also applies to integral operator on $L^p$-functions. So $P$ commutes with $\partial_{\nu}$.

Note $P\circ S=P\circ T-P\circ P\circ T=0$. Since $P$ preserves $A^p(D_1)$, $P\circ[\partial_{\nu},S](g)=[\partial_{\nu},S](g)$ for a.e. $z^2\in D_2$. Therefore
\[
[\partial_{\nu},S](g)=P\circ[\partial_{\nu},S](g)=P\left(\partial_{\nu}S(g)-S(\partial_{\nu}g)\right)=\partial_{\nu}PS(g)-PS(\partial_{\nu}g)=0
\]
for a.e. $z^2\in D_2$. This completes the proof.
\end{proof}

\begin{corollary}
\label{solutioncommute}
Let $p$, $D_1$, and $D_2$ be as in Lemma \ref{differentialquotient}. For $j=1,2$, let $T_j:L^p_{(0,1)}(D_j)\to L^p(D_j)$ be a bounded linear operator, which solves the $\dbar_j$-equation on $D_j$, and let $P_j:L^p(D_j)\to A^p(D_j):=L^p(D_j)\cap\co(D_j)$ be a projection operator, which preserves $A^p(D_j)$. Then for $j=1,2$, $S_j:=T_j-P_j\circ T_j:L_{(0,1)}^p(D_j)\to L^p(D_j)$ also solves the $\dbar_j$-equation on $D_j$. Moreover, if $g:=\dbar_2\pi_1(f)=-\dbar_1\pi_2(f)\in L^p_{(0,2)}(D_1\times D_2)$, where $f\in L^p_{(0,1)}(D_1\times D_2)$ is $\dbar$-closed on $D_1\times D_2$, then $[S_1,S_2](g)=0$.
\end{corollary}

\begin{proof}
It suffices to show commutativity of $S_1$ and $S_2$ on $g$. Note that $[S_1,S_2](g)\in L^p(D_1\times D_2)$ and
\[
\dbar_1\circ[S_1,S_2](g)=\dbar_1(S_1S_2(g)-S_2S_1(g))=S_2(g)-S_2\dbar_1S_1(g)=0,
\]
since $\dbar_1$ commutes with $S_2$ by Lemma \ref{commutativity}. Thus $[S_1,S_2](g)(\cdot,z^2)\in A^p(D_1)$ for a.e. $z^2\in D_2$. Similarly, $[S_1,S_2](g)(z^1,\cdot)\in A^p(D_2)$ for a.e. $z^1\in D_1$.

Note that $P_j$ preserves $A^p(D_j)$ and $P_j\circ S_j=0$ for $j=1,2$. Since $P_1$ and $P_2$ can be represented as integral operators, $P_1\circ P_2=P_2\circ P_1$. Therefore
\begin{align*}
[S_1,S_2](g)
&=P_2P_1[S_1,S_2](g)=P_2P_1(S_1S_2(g)-S_2S_1(g))\\
&=-P_2P_1S_2S_1(g)=-P_1P_2S_2S_1(g)=0.
\end{align*}
\end{proof}

Summarizing the last two lemmas: if there exist $L^p$-bounded solution and projection operators, there exists an $L^p$-bounded solution that commutes with directional derivatives in other factors of the product domain. This can be applied to several different cases including non-pseudoconvex domains. Here are two examples when $p=2$.

\begin{example}
Let $D_1$, $D_2$, and $\partial_{\nu}$ be as in Lemma \ref{commutativity}. Let $T:L_{(0,1)}^2(D_1)\to L^2(D_1)$ be a bounded linear operator, which solves the $\dbar_1$-equation on $D_1$. Let $P:L^2(D_1)\to A^2(D_1)$ be the Bergman projection. Then the solution operator $S:=T-P\circ T$ is $L^2$-bounded on $D_1$ and satisfies $[\partial_{\nu},S](g)=0$ for any $g\in L^2_{(0,1)}(D_1\times D_2)$, $\partial_{\nu}g\in L^2_{(0,1)}(D_1\times D_2)$, and $\dbar_1(g)=0$.

In particular, if $\dbar_1$ has closed range, the existence of $T$ is guaranteed by taking $T$ to be the $L^2$-canonical solution operator for $\dbar_1$ on $D_1$. In this case $S=T$, since $P\circ T=0$.
\end{example}

\begin{example}
\label{Kohnsolution}
Let $D_1$, $D_2$, and $\partial_{\nu}$ be as in Lemma \ref{commutativity}. Assume further that $D_1$ is pseudoconvex. For $t>0$, let $\mu_t(z^1)=e^{-t|z^1|^2}$ be a weight on $D_1$. By Kohn's weighted $L^2$ theory \cite{Kohn73}, the weighted $L^2$-canonical solution $T_{\mu_t}:L_{(0,1)}^2(D_1,\mu_t)\to L^2(D_1,\mu_t)$. Let $P_{\mu_t}:L^2(D_1,\mu_t)\to A^2(D_1,\mu_t)$ be the weighted Bergman projection. Then $P_{\mu_t}\circ T_{\mu_t}=0$ and hence $S_{\mu_t}=T_{\mu_t}$.

Given $t>0$, the weight $\mu_t(z^1)\approx 1$ is comparable to a constant. So for each $t>0$, $L^2(D_1,\mu_t)=L^2(D_1)$ and $T_{\mu_t}$ and $P_{\mu_t}$ are $L^2$-bounded on the unweighted spaces. Therefore $[\partial_{\nu},T_{\mu_t}](g)=0$ for any $g\in L^2_{(0,1)}(D_1\times D_2)$, $\partial_{\nu}g\in L^2_{(0,1)}(D_1\times D_2)$, and $\dbar_1(g)=0$.
\end{example}

\subsection{Sobolev estimates}
The basic argument in Propositions \ref{highdimexistence}, \ref{LpforhighT} yields  $L^2$-Sobolev regularity of a solution operator for $\dbar$ on product spaces. Let $W^{s}(D)=W^{s,2}(D)$ denote the $L^2$-Sobolev space on $D$.

\begin{theorem}
\label{sobolev}
Assume that each factor of $D=D_1\times\cdots\times D_k$ is bounded pseudoconvex with $C^{\infty}$ boundary. 
For each $m\in\Z^+$,  there exists a bounded linear operator \newline $S_m:W_{(0,1)}^{m+k-1}(D)\cap\ker(\dbar)\to W^{m}(D)$ satisfying
\[
\dbar\big(S_m f\big)=f\quad\text{on } D,
\]
for all  $f\in W_{(0,1)}^{m+k-1}(D)$ satisfying $\dbar f=0$.
\end{theorem}

\begin{proof}
For each $j=1,\dots,k$, let $T_{j,\mu_t}$ be the weighted $L^2$-canonical solution for $\dbar_j$ on $D_j$ as in Example \ref{Kohnsolution}. Pick $t>0$ sufficiently large. By \cite{Kohn73}, $T_{j,\mu_t}$ is $W^{m+k-1}(D_j)$-bounded, see also \cite[Theorem 5.1 in \S 5.1]{Straube10}. Since $T_{j,\mu_t}$ is also $L^2(D_j)$-bounded, by interpolation of $L^2$-Sobolev spaces, $T_{j,\mu_t}$ is $W^s(D_j)$-bounded for $0\le s\le m+k-1$. By Example \ref{Kohnsolution}, for each $j$, $T_{j,\mu_t}$ commutes with $\partial_{\nu}$ on $W^1_{(0,1)}(D)\cap\ker(\dbar_j)$ for any direction $\nu$ not in $D_j$. 

Since $f\in W_{(0,1)}^{m+k-1}(D)$, $f\in\sB^2$; so Proposition \ref{highdimexistence} applies for $p=2$. It remains to verify the $(0,1)$-forms $g_j$ constructed in the proof of Proposition \ref{highdimexistence} are in $L^2_{(0,1)}(D)$ and the $L^2$-Sobolev regularity of $S_m(f):=v^1+\cdots+v^k$ as in Proposition \ref{LpforhighT}.

By construction, $v^j=T_{j,\mu_t}(\pi_j(g_j))$. Example \ref{Kohnsolution} shows $T_{j,\mu_t}$ commutes with directional derivatives $\partial_{\nu}$ when $\nu$ is not in a direction given by $D_j$. If $\nu$ points in a direction of $D_j$,  
$\left\|\partial_{\nu}v^j\right\|_s \lesssim \left\|\pi_j(g_j)\right\|_{s+1}$ since $T_{j,\mu_t}$ on $W^s(D_j)\to W^s(D_j)$. Hence
\[
\|v^j\|_{W^s(D)}\le C\|\pi_j(g_j)\|_{W^s_{(0,1)}(D)}\le C\|g_j\|_{W^s_{(0,1)}(D)}.
\]
Note that $g_j=g_{j-1}+\dbar v^{j-1}$. So
\[
\|g_j\|_{W^s_{(0,1)}(D)}\le\|g_{j-1}\|_{W^s_{(0,1)}(D)}+\|v^{j-1}\|_{W^{s+1}(D)}.
\]
Combining these recursive estimates yields
\[
\|v^j\|_{W^m(D)}\le C\|g_1\|_{W^{j+m-1}_{(0,1)}(D)}=C\|f\|_{W^{j+m-1}_{(0,1)}(D)}
\]
and
\[
\|g_j\|_{W^m_{(0,1)}(D)}\le C\|g_1\|_{W^{j+m-1}_{(0,1)}(D)}=C\|f\|_{W^{j+m-1}_{(0,1)}(D)},
\]
for $j=1,\dots,k$. Therefore $g_j\in L^2_{(0,1)}(D)$ and
\[
\|S_m(f)\|_{W^m(D)}\le\sum_{j=1}^k\|v^j\|_{W^m(D)}\le C\|f\|_{W^{m+k-1}_{(0,1)}(D)}.
\]
\end{proof}

\begin{remark}
This improves the $L^2$-Sobolev estimate in Section 6.2 of \cite{ChakrabartiShaw11}. Also, when $f\in C^{\infty}_{(0,1)}(\overline{D})$, a standard Mittag-Leffler construction, see \cite{Straube10} or \cite{ChenShaw01}, yields a solution $u\in\cap_{m=1}^{\infty}W^m(D)$. Since $D$ is a product domain with smooth factors, $D$ satisfies the strong local Lipschitz condition---each boundary point has a neighborhood such that the boundary in that neighborhood is the graph of a Lipschitz function. Thus $u\in C^{\infty}(\overline{D})$ by the Sobolev embedding theorem.
\end{remark}

\begin{remark}
By Lemma \ref{commutativity} and Corollary \ref{solutioncommute}, the solution operators $T_{j,\mu_t}$ commute with each other and with $\dbar_i$ when $i\neq j$. Consequently, $S_m(f)=v^1+\cdots+v^k$ is independent of the order of the factors in $D=D_1\times\cdots\times D_k$. This follows by employing the observation \eqref{E:alt2}. This also holds for the solution in Corollary \ref{C:Wkp}, where commutativity of the solution operators on each factor is guaranteed by Fubini's Theorem.
\end{remark}


\section{Higher dimensional factors: H\"{o}lder estimates}
\label{S:Holder}

\subsection{The H\"{o}lder spaces} In Sobolev or $L^p$ spaces, norms can be evaluated by iterating lower dimensional integrals. In standard H\"{o}lder spaces this is not possible, since the difference $z-z'$ can be in any direction.

\begin{definition}
For $0<\beta<1$, the classical H\"{o}lder space of order $\beta$ on the domain $D$ is the set
\[
\Lambda^{\beta}(D)=\{g\in L^{\infty}|\,\|g\|_{\Lambda^{\beta}(D)}<\infty\}
\]
with the norm $\|g\|_{\Lambda^{\beta}(D)}=\|g\|_{L^{\infty}(D)}+\|g\|_{\beta}$, where
\[
\|g\|_{\beta}=\sup_{z\neq z'\in D}\frac{|g(z)-g(z')|}{|z-z'|^{\beta}}.
\]
\end{definition}

The following example is due to Stein.

\begin{example}
Consider the bidisk $\D^2$ in $\C^2$. Let $v(z_1,z_2)=\bar z_2/\log(z_1-1)$ be a complex function on $\D^2$, where $\pi/2<\arg(z_1-1)<3\pi/2$. Let $f=\dbar v$ be a $(0,1)$-form on $\D^2$.

Then $f$ is $\dbar$-closed and $f\in L^{\infty}_{(0,1)}(\D^2)$. For $\beta>0$, there is no complex function $u$ on $\D^2$ such that $\dbar u=f$ and $u\in\Lambda^{\beta}(\D^2)$. See \cite[section 1.2 B]{Kerzman71} for details.
\end{example}

This example shows classical H\"{o}lder norms are not really suited to the $\dbar$-equation on product domains. A substitute scale of iterated H\"{o}lder spaces is considered here. For simplicity, let $D=D_1\times D_2$, where $D_1\subset\C^{n_1}$ and  $D_2\subset\C^{n_2}$ are bounded domains.

\begin{definition}
\label{2sepHolder}
For $\alpha=(\alpha_1,\alpha_2)$ with $0<\alpha_1,\alpha_2<1$, the $2$-iterated H\"{o}lder space of order $\alpha$ is the set
\[
\Lambda_{2}^{\alpha}(D)=\{g\in L^{\infty}(D)\,|\,\|g\|_{\Lambda_{2}^{\alpha}(D)}<\infty\}
\]
with the norm $\|g\|_{\Lambda_{2}^{\alpha}(D)}=\|g\|_{L^{\infty}(D)}+\|g\|_{\alpha}$, where
\[
\|g\|_{\alpha}=\sup_{z^1\neq w^1 \in D_1}\sup_{z^2\neq w^2\in D_2}\frac{|g(z^1,z^2)-g(w^1,z^2)-g(z^1,w^2)+g(w^1,w^2)|}{|z^1-w^1|^{\alpha_1}|z^2-w^2|^{\alpha_2}}.
\]
\end{definition}

\begin{remark}
Note that $\|g\|_{\alpha}$ is independent of the order of taking the $\sup$ in $D_1$ and $D_2$, i.e.
\[
\sup_{z^2\neq w^2\in D_2}\frac{\|g(\cdot,z^2)-g(\cdot,w^2)\|_{\alpha_1(D_1)}}{|z^2-w^2|^{\alpha_2}}=\sup_{z^1\neq w^1 \in D_1}\frac{\|g(z^1,\cdot)-g(w^1,\cdot)\|_{\alpha_2(D_2)}}{|z^1-w^1|^{\alpha_1}}.
\]
\end{remark}

The following examples show iterated H\"{o}lder spaces are quite different from classical H\"{o}lder spaces.
\begin{example}
Let $g(z_1,z_2)=1/\log(z_1-1)$ on $\D^2$, where $\pi/2<\arg(z_1-1)<3\pi/2$.
\begin{enumerate}
\item It is not hard to check $g\in L^{\infty}(\D^2)$.
\item Take $z=(1-2d,0)$ and $z'=(1-d,0)$, where $0<d<1/2$ is sufficiently small. Since
\[
\left|\frac{1}{\log(-2d)}-\frac{1}{\log(-d)}\right|\nleq Cd^{\beta}
\]
for any $0<\beta<1$ and constant $C>0$, we see that $g\notin\Lambda^{\beta}(\D^2)$. 
\item For $\alpha=(\alpha_1,\alpha_2)$ with $0<\alpha_1,\alpha_2<1$, $g\in\Lambda^{\alpha}_2(\D^2)$, since for any $z_1,z_2,w_1,w_2\in\D$
\[
g(z_1,z_2)-g(w_1,z_2)-g(z_1,w_2)+g(w_1,w_2)\equiv 0.
\]
\end{enumerate}
\end{example}

\begin{example}
On $\D^2$, let
\[
g(z_1,z_2)=\left\{ \begin{array}{lc} |z|^2\sin\left(\frac{1}{|z|}\right), & z\neq 0 \\ 0, & z=0 \end{array}\right. \qquad \text{where }|z|^2=|z_1|^2+|z_2|^2.
\]
\begin{enumerate}
\item It is not hard to check $g\in L^{\infty}(\D^2)$.
\item Since $g$ is real differentiable and its differential is bounded, $g\in\Lambda^{\beta}(\D^2)$ for $0<\beta\le 1$ by the mean value theorem.
\item Given any $\alpha=(\alpha_1,\alpha_2)$ with $0<\alpha_1,\alpha_2<1$, take $z_1=a>0$, $z_2=b>0$, and $w_1=w_2=0$, then
\[
\|g\|_{\alpha}\ge\sup_{a,b\in(0,1)}\frac{\left|(a^2+b^2)\sin\frac{1}{\sqrt{a^2+b^2}}-b^2\sin\frac{1}{b}-a^2\sin\frac{1}{a}\right|}{a^{\alpha_1}b^{\alpha_2}}=M.
\]
If $b=a^k$, where $k$ is sufficiently large, then $M\ge \lim_{a\to0^+}Ca^{2-\alpha_1-k\alpha_2}=\infty$, provided $2<\alpha_1+k\alpha_2$. So $g\notin\Lambda^{\alpha}_2(\D^2)$.
\end{enumerate}
\end{example}

\subsection{The estimates}
Consider the $\dbar$-equation on $D=D_1\times D_2$
\begin{equation}
\label{dbarD1D2}
\dbar u=f=\pi_1(f)+\pi_2(f),
\end{equation}
for $f$ a $(0,1)$-form satisfying $\dbar f=0$ weakly. Notationally $z^j\in D_j$,
\[
\pi_1(f)=\bm{f}^1=f^1_1\,d\bar z^1_1+\cdots+f^1_{n_1}\,d\bar z^1_{n_1},\qquad\text{and}\qquad \pi_2(f)=\bm{f}^2=f^2_1\,d\bar z^2_1+\cdots+f^2_{n_2}\,d\bar z^2_{n_2}.
\]

\begin{theorem}
\label{main3}
Consider the $\dbar$-equation \eqref{dbarD1D2}. For $j=1,2$, assume there is a linear bounded operator $T_j:L^{\infty}_{(0,1)}(D_j)\to\Lambda^{\alpha_j}(D_j)$ solving  the $\dbar_j$-equation that commutes with  the barred derivatives on the other factor; generically denote these as $\partial/\partial\bm{\bar z}_*$.

If $f_{\{1,2\}}\in L^{\infty}(D)$, $f_{\{1\}}\in L^{\infty}(D_1)\otimes\Lambda^{\alpha_2}(D_2)$, and $f_{\{2\}}\in\Lambda^{\alpha_1}(D_1)\otimes L^{\infty}(D_2)$, there exists a solution $u=T(f)$ of \eqref{dbarD1D2} satisfying
\[
\|u\|_{\Lambda_{2}^{\alpha}(D)}\le C_{\alpha}\left(\|f_{\{1\}}\|_{L^{\infty}(D_1)\otimes\Lambda^{\alpha_2}(D_2)}+\|f_{\{2\}}\|_{\Lambda^{\alpha_1}(D_1)\otimes L^{\infty}(D_2)}+\|f_{\{1,2\}}\|_{L^{\infty}(D)}\right).
\]  
Recall $0<\alpha_1,\alpha_2<1$ above.
\end{theorem}

\begin{corollary}
\label{Holderestimate}
In particular, if $f_I\in\Lambda_{2}^{\alpha}(D)$ for all $I\neq0$, the solution $u=T(f)$ satisfies the estimate
\[
\|u\|_{\Lambda_{2}^{\alpha}(D)}\le C_\alpha\sum_{|I|\neq0}\|f_I\|_{\Lambda_{2}^{\alpha}(D)}.
\]
The operator $T$ is bounded from $\{f\,\dbar\text{-closed}\,|\,f_I\in\Lambda_{2}^{\alpha}(D)\}$ into $\Lambda_{2}^{\alpha}(D)$. 
\end{corollary}

\begin{proof}[Proof of Theorem \ref{main3}]
This follows the same argument as the proof of Theorem \ref{main2}. Let $g_1=f$. For each point in $D_2$, let $v^1=T_1(\pi_1(g_1))$. Then $v^1$ satisfies the $\dbar_1$-equation on $D_1$
\[
\dbar_1 v^1=\pi_1(g_1).
\]
Take any two distinct points $z^2,w^2\in D_2$ and consider the difference $v^1(\cdot,z^2)-v^1(\cdot,w^2)$;  this satisfies 
\[
\dbar_1\big(v^1(\cdot,z^2)-v^1(\cdot,w^2)\big)=\pi_1(g_1)(\cdot,z^2)-\pi_1(g_1)(\cdot,w^2).
\]
By regularity and linearity of $T_1$, it follows that
\[
\|v^1(\cdot,z^2)\|_{L^{\infty}(D_1)}\le \|v^1(\cdot,z^2)\|_{\Lambda^{\alpha_1}(D_1)}\le C_{\alpha_1}\|\pi_1(g_1)(\cdot,z^2)\|_{L^{\infty}_{(0,1)}(D_1)}
\]
and
\begin{align*}
\|v^1(\cdot,z^2)-v^1(\cdot,w^2)\|_{\alpha_1(D_1)}&\le\|v^1(\cdot,z^2)-v^1(\cdot,w^2)\|_{\Lambda^{\alpha_1}(D_1)} \\ &\le C_{\alpha_1}\|\pi_1(g_1)(\cdot,z^2)-\pi_1(g_1)(\cdot,w^2)\|_{L^{\infty}_{(0,1)}(D_1)}.
\end{align*}
This yields the iterated H\"{o}lder estimate for $v^1$:
\begin{equation}
\label{forv1}
\|v^1\|_{\Lambda_{2}^{\alpha}(D)}\le C_{\alpha_1}\|f_{\{1\}}\|_{L^{\infty}(D_1)\otimes\Lambda^{\alpha_2}(D_2)}.
\end{equation}

Let $g_2=g_1-\dbar v^1$ on $D$. As in the proof of Proposition \ref{highdimexistence}, let $v^2=T_2(\pi_2(g_2))$. Since $\pi_1(g_2)=\pi_1(g_1)-\dbar_1 v^1=0$, it holds that $g_2=\pi_2(g_2)=\pi_2(g_1)-\dbar_2 v^1$. By  regularity and  linearity of $T_2$, the argument that gave \eqref{forv1} shows 
\begin{equation}
\label{forv2}
\|v^2\|_{\Lambda_{2}^{\alpha}(D)}\le C_{\alpha_2}\Big(\|f_{\{2\}}\|_{\Lambda^{\alpha_1}(D_1)\otimes L^{\infty}(D_2)}+\|\dbar_2 v^1\|_{\Lambda^{\alpha_1}(D_1)\otimes L^{\infty}_{(0,1)}(D_2)}\Big).
\end{equation}

As in the proof of Proposition \ref{LpforhighT}, commutativity of $T_1$ and $\partial/\partial\bar z^2_{j}$ and regularity of $T_1$, for $j=1,\dots,n_2$ implies
\[
\left\|\frac{\partial v^1}{\partial \bar z^2_{j}}(\cdot,z^2)\right\|_{\Lambda^{\alpha_1}(D_1)}\le C_{\alpha_1}\left\|\frac{\partial}{\partial \bar z^2_{j}}\pi_1(g_1)(\cdot,z^2)\right\|_{L^{\infty}_{(0,1)}(D_1)}.
\]
Thus
\begin{equation}
\label{fordv1}
\|\dbar_2 v^1\|_{\Lambda^{\alpha_1}(D_1)\otimes L^{\infty}_{(0,1)}(D_2)}\le C_{\alpha_1}\|f_{\{1,2\}}\|_{L^{\infty}(D)}.
\end{equation}

The function $u=T(f)=v^1+v^2$ solves $\dbar u=f$, just as in the proof of Proposition \ref{highdimexistence}. The iterated H\"{o}lder estimate for $u$ is obtained by combining \eqref{forv1}, \eqref{forv2}, and \eqref{fordv1}. \end{proof}

\begin{corollary}\label{C:iteratedholder}
Let $D=D_1\times D_2$, where $D_j$ is a bounded domain in $\C^{n_j}$ for $j=1,2$. If $D_j$ is $1$-dimensional, then $\alpha_j$ can be any value in $(0,1)$. If $D_j$ is at least $2$-dimensional and is strongly pseudoconvex with $C^2$ boundary, then $\alpha_j$ can take any value in $(0,1/2]$. 

Assume $f_{\{1,2\}}\in L^{\infty}(D)$, $f_{\{1\}}\in L^{\infty}(D_1)\otimes\Lambda^{\alpha_2}(D_2)$, and $f_{\{2\}}\in\Lambda^{\alpha_1}(D_1)\otimes L^{\infty}(D_2)$. There exists a solution $u=T(f)$ of the equation \eqref{dbarD1D2} with estimate
\[
\|u\|_{\Lambda_{2}^{\alpha}(D)}\le C_{\alpha}\left(\|f_{\{1\}}\|_{L^{\infty}(D_1)\otimes\Lambda^{\alpha_2}(D_2)}+\|f_{\{2\}}\|_{\Lambda^{\alpha_1}(D_1)\otimes L^{\infty}(D_2)}+\|f_{\{1,2\}}\|_{L^{\infty}(D)}\right).
\]
In particular, if $f_I\in\Lambda_{2}^{\alpha}(D)$ for all $I\neq0$, then the solution $u=T(f)$ satisfies the estimate
\[
\|u\|_{\Lambda_{2}^{\alpha}(D)}\le C_\alpha\sum_{|I|\neq0}\|f_I\|_{\Lambda_{2}^{\alpha}(D)}.
\]  
\end{corollary}

\begin{proof}
If the factors are one-dimensional, the conclusion follows since
\[
-\frac{1}{2\pi i}\int_{D_0}\frac{g(\zeta)\,d\bar\zeta\wedge d\zeta}{\zeta-z}
\]
maps $L^{\infty}$ functions to $\Lambda^{\beta}$ functions for all $0<\beta<1$. See, for example, the proof of Lemma 1.15 in \cite[Chapter IV \S1.6]{Range86}.

If $D_j$  is  strongly pseudoconvex with $C^2$ boundary, then \cite[Chapter V \S2.4, Theorem 2.7]{Range86} guarantees existence and regularity of the solution operator on $D_j$. Since the solution operator is given by the sum and composition of integral operators, Lemma \ref{integralcommute} guarantees the commutativity needed for  Theorem \ref{main3}. 
\end{proof}

\begin{remark}
Let $D=D_1\times\cdots\times D_k$, where $D_j$ is a bounded domain in $\C^{n_j}$ for $j=1,\dots,k$ and let $\alpha=(\alpha_1,\dots,\alpha_k)$ with $0<\alpha_1,\dots,\alpha_k<1$. The $k$-iterated H\"{o}lder spaces $\Lambda^{\alpha}_{k}(D)$ are defined in a similar fashion to Definition \ref{2sepHolder}. Iterated H\"{o}lder estimates on $\dbar$ for products with more than two factors follow by invoking the arguments above, but are left to the interested reader.
\end{remark}

\begin{remark}
Like Corollary \ref{C:Lp}, Corollary \ref{C:iteratedholder} partially answers a question discussed earlier: is there a solution operator for $\dbar$ on $D_1\times D_2$ that preserves $L^{\infty}$? As noted in the Introduction,
a priori smoothness on $f$  is required before $\|u\|_{L^\infty}\lesssim \|f\|_{L^\infty}$ can be derived in Corollary \ref{C:Lp}, Corollary \ref{C:iteratedholder}, or the earlier  \cite{FornaessLeeZhang, CheMcN18}.

However in Corollary \ref{C:iteratedholder}, a little extra holds: from the assumption  that $f_I\in L^{\infty}$,  iterated H\"{o}lder regularity of $Tf$ is obtained. This is similar to the gain in $L^p$ integrability noticed in Remark \ref{R:Lp}. 
\end{remark}


\section{Orthogonality}\label{S:ortho}

Return to products with one-dimensional factors: $D=D_1\times\cdots\times D_n$, where $D_j\subset\C$ are domains with piecewise $C^1$ boundary. Proposition \ref{integralformula} implies \eqref{CIf}  solves  $\dbar v=f$ if there exists a $u\in C^n(\overline{D})$ solving $\dbar u=f$ on $D$. Moreover the left hand side of \eqref{CIf} shows that assuming $u\in C(\overline{D})$ suffices.

Note that if $u',u''\in C(\overline{D})$ both weakly solve the $\dbar$-equation, then
$u'-u''\in \co(D)\cap C(\overline{D})$.
Since $\sC_n$ preserves holomorphic functions, uniqueness of $T$  follows:
\[
\text{id}(u')-\sC_n(u')=\text{id}(u'')-\sC_n(u'').
\]
In other words, $T$ is independent of the auxiliary solution $u$ in Proposition \ref{integralformula}.
In particular, if the solution in Theorem \ref{main1} is in $C(\overline{D})$, uniqueness implies
$\text{id}(u)-\sC_n(u)=u$, which in turn implies
\[
\sC_n(u)=0.
\]

This argument yields 

\begin{theorem}
\label{Cinfty}
Let $D$ be a product domain with one-dimensional bounded smooth factors.  If $f\in C_{(0,1)}^{\infty}(\overline{D})$ and $\dbar f=0$, the solution $Tf=-\sum_{\emptyset\neq I\subset\{1,\dots,n\}}\bm{C}^{I}(f_{I}^{I^c})\in C^{\infty}(\overline{D})$,. Moreover, the multi-Cauchy transform of $Tf$ is $0$.
\end{theorem}

\begin{proof}
This follows from the argument above, after repeatedly applying Lemma \ref{L:smoothCT} to the expression for $T$.
\end{proof}

\begin{remark}
By examining the proof of Lemma \ref{L:smoothCT} and the proof of Theorem \ref{Cinfty}, one can relax the smoothness requirement of $f$ and $D$ to obtain a solution $u=T(f)\in C(\overline{D})$, which also satisfies $\sC_n(u)=0$. \end{remark}

In particular, if $D=\D^n$, then $\cs_n=\sC_n$ is the Szego projection onto $H^2((b\D)^n)$, the Hardy space of the distinguished boundary.

\begin{corollary}\label{C:hardy}
Let $f\in C_{(0,1)}^{\infty}(\overline{\D^n})$ satisfy $\dbar f=0$. Then $u=T(f)\in C^{\infty}(\overline{\D^n})$ and $u$ is orthogonal to the Hardy space $H^2((b\D)^n)$.
\end{corollary}

\begin{remark}
The canonical or Kohn solution to $\dbar$ is orthogonal to the Bergman space. The Corollary \ref{C:hardy}  says the solution $Tf$ on $\D^n$ is orthogonal to the Hardy space $H^2((b\D)^n)$. The usual argument then imply $Tf$ is the $L^2((b\D)^n)$-minimal solution.
\end{remark}

Let $D$ be a product domain with higher dimensional factors. The regularity results in \S\ref{smoothregularity} show that if $f$ is sufficiently smooth and $m$ is sufficiently large, the solution $u=S_m(f)$ to $\dbar u=f$ belongs
to $C(\overline{D})$. The argument above applies to $v=u-\cs(u)=(\text{id}-\cs)\circ T(f)$ for any projection $\cs:L^2(\Gamma)\to H^2(\Gamma):=L^2(\Gamma)\cap\co(D)$ which preserves $H^2(\Gamma)$, where $\Gamma$ is a subset of the full boundary $bD$. 

There are several natural subsets $\Gamma\subset bD$, since $D$ is a product domain. Thus there are several  ``Szeg\" o projections'' on product domains. A representative result is

\begin{corollary}
Let $D=D_1\times\cdots\times D_k$ be as in Theorem \ref{sobolev}. 

For each $j=1,\dots,k$, let $\cs_j:L^2(bD_j)\to H^2(D_j)$ be the Szeg\"{o} projection on $D_j$. Let $\Gamma=bD_1\times\cdots\times bD_k$ be the distinguished boundary of $D$. Define the orthogonal projection $\cs=\cs_1\otimes\cdots\otimes\cs_k:L^2(\Gamma)\to H^2(D)$ to be the Szeg\"{o} projection on the $\Gamma$. 

For $m$ sufficiently large, let $S_m$ be the solution operator obtained in Theorem \ref{sobolev}.
Then $v=(\text{id}-\cs)\circ S_m(f)$ is also a solution to
\[
\dbar u=f,
\]
where $f\in W^{m+k-1}_{(0,1)}(\overline{D})$ is $\dbar$-closed. Moreover, $v$ is orthogonal to $H^2(\Gamma)$, the Hardy space associated to $\Gamma$.
\end{corollary}

Passing from a solution $\dbar u =f$ to another solution of the same equation is a powerful tool in complex analysis. In practice the new solution is constructed to satisfy additional properties, which depend on
the problem at hand. A summary of such changes of solutions made in the paper is presented, to suggest further application. Let
$T$ be a linear operator solving $\dbar(Tf)=f$. 
\smallskip

\begin{enumerate}
\item $v=(\text{id}-\sC_n)\big(Tf\big)$ is another solution satisfies $\sC_n(v)=0$ with other regularity properties.
\smallskip
\item $w=(\text{id}-\cs)\circ S_m(f)$ is a solution which is annihilated by the projection $\cs$.
\smallskip
\item $S=(\text{id}-P)\circ T$ is a solution operator which commutes with directional differentiation on other factors and inherits regularity from $T$.
\end{enumerate}


\section{Appendix: Cauchy transforms in $\C$}\label{S:appendix}

Results about the one-variable operators $\sC$ and $\bm{C}$ used in previous sections are gathered here. These results are not new but also not sufficiently well-known. The results do not appear in standard texts with the exception of Lemma \ref{L:stokes}. The second author learned these results from S.R. Bell \cite{BellCourse}.

First recall the definitions.
If $D_0\subset\C$ is a bounded domain with  piecewise $C^1$ boundary $bD_0$ and $g\in C(bD_0)$, the Cauchy transform of $g$ is defined
\begin{equation}\label{E:CT1}
\sC(g)(z)=\frac{1}{2\pi i}\int_{bD_0}\frac{g(\zeta)\,d\zeta}{\zeta-z}.
\end{equation}
 If $h\in C(D_0)$, the {\it solid Cauchy transform of} $h$ is defined
\begin{equation}\label{E:solidCT2}
\bm{C}(h)(z)=\frac{1}{2\pi i}\int_{D_0}\frac{h(\zeta)\,d\bar\zeta\wedge d\zeta}{\zeta-z}\qquad h\in C(D_0).
\end{equation}

\subsection{Smoothness to the boundary; Bell's method}

The solid Cauchy transform $\bm{C}$ is not immediately seen to preserve $C^{\infty}(\overline{D_0})$.  However the following idea, due to Bell \cite{BellLigocka, Bell81a}, shows this holds.

\begin{definition}
Let $D_0\subset\C$ be a smoothly bounded domain with defining function $r$. Let $M\in\Z^+$ and $F,G\in C^M(\overline{D_0})$. Say $F=G$ on $bD_0$ to order $M$ if there exists $H\in C^{\infty}(\overline{D_0})$ satisfying
\[
F-G=H\cdot r^M.
\]
\end{definition}

\begin{lemma} \label{L:Bell}
If  $D_0\subset\C$ is a smoothly bounded domain, $h\in C^{\infty}(\overline{D_0})$, and $M\in\Z^+$, then there exists $h^M\in C^M(\overline{D_0})$ so that
\begin{enumerate}
\item[(i)] $h^M=0$ on $bD_0$ to order $1$;
\item[(ii)] $h=\frac{\partial h^M}{\partial \bar z}$ on $bD_0$ to order $M$.
\end{enumerate}
\end{lemma}

\begin{proof}
This is in the spirit of the proof of Bell's lemma, cf.  \cite{BellLigocka, Bell81a}. Let $r$ be a defining function of $D_0$, so that $dr\neq0$ when $r=0$. The functions $h^M$ are constructed by induction on $M$.

Let $M=1$.  Set $h^1=\phi_1\cdot r$ for $\phi_1\in C^{\infty}(\overline{D_0})$ to be determined. Any such $\phi_1$ implies $h^1$ satisfies (i). However
\[
\frac{\partial h^1}{\partial \bar z}-h=(\phi_1)_{\bar z}\cdot r+\phi_1\cdot r_{\bar z}-h.
\]
Thus,  taking $\phi_1=h/r_{\bar z}$ causes $h^1$ to satisfy (ii) as well.

Let $M=k+1$ and assume that $h^{k}$ has been constructed. Let $h^{k+1}$ be of the form $h^{k+1}=h^k+\phi_{k+1}\cdot r^{k+1}$, for $\phi_{k+1}\in C^{\infty}(\overline{D_0})$ to be determined. A computation gives
\begin{align*}
\frac{\partial h^{k+1}}{\partial \bar z}-h
&=\frac{\partial h^{k}}{\partial \bar z}-h+(\phi_{k+1})_{\bar z}\cdot r^{k+1}+\phi_{k+1}\cdot (k+1)r^kr_{\bar z}\\
&=(\phi_{k+1})_{\bar z}\cdot r^{k+1}+[\psi_k\cdot r^k+\phi_{k+1}\cdot (k+1)r^kr_{\bar z}],
\end{align*}
where $\psi_k\in C^{\infty}(\overline{D_0})$ by the induction hypothesis. Taking $\phi_{k+1}=-\psi_k/[(k+1)r_{\bar z}]$ causes $h^{k+1}$ to satisfy (ii) for $M=k+1$. 
\end{proof}

\noindent Smoothness results on $\bm{C}(h)$ follows from Lemma \ref{L:Bell}:

\begin{lemma}\label{L:smoothCT}
If $D_0\subset\C$ is a smoothly bounded domain and $h\in C^{\infty}(\overline{D_0})$, then
\[
\bm{C}(h)(z)\in C^{\infty}(\overline{D_0}).
\]
\end{lemma}

\begin{proof} Let $M\in\Z^+$ be given. Apply Lemma \ref{L:Bell} to $h$ to obtain a function $h^M$ satisfying the conclusion of the lemma. Lemma \ref{L:stokes} applied to $h^M$ yields
\[
h^M(z)=\frac{-1}{2\pi i}\int_{D_0}\frac{\partial h^M}{\partial \bar\zeta}(\zeta)\cdot\frac{d\bar\zeta\wedge d\zeta}{\zeta-z}.
\]
Therefore
\begin{align*}
\bm{C}(h)(z)+h^M(z)
&=\frac{1}{2\pi i}\int_{D_0}[h(\zeta)-\frac{\partial h^M}{\partial \bar\zeta}(\zeta)]\cdot\frac{d\bar\zeta\wedge d\zeta}{\zeta-z}\\
&=\frac{1}{2\pi i}\int_{D_0}\sd_M(\zeta)\cdot\frac{d\bar\zeta\wedge d\zeta}{\zeta-z},
\end{align*}
where $\sd_M=h-\partial h^M/\partial \bar z$. Since $\sd_M$ vanishes to order $M$ on $bD_0$, $\sd_M\in C_c^M(\C)$ by setting $\sd_M=0$ outside $D_0$. Thus
\[
\bm{C}(h)(z)+h^M(z)=\frac{1}{2\pi i}\int_{\C}\sd_M(\zeta)\cdot\frac{d\bar\zeta\wedge d\zeta}{\zeta-z}=\left(\sd_M*\frac{1}{\zeta}\right)(z).
\]
However $\sd_M*1/\zeta\in C^M(\C)$, since $\sd_M$ has compact support.  Since $h^M\in C^M(\overline{D_0})$ as well, it follows that $\bm{C}(h)\in C^M(\overline{D_0})$. $M$ was arbitrary, so $\bm{C}(h)\in C^{\infty}(\overline{D_0})$ follows.
\end{proof}

\subsection{A solution operator}
Stokes theorem connects the Cauchy and solid Cauchy transforms.

\begin{lemma}\label{L:stokes} If $D_0\subset\C$ is a bounded domain with  piecewise $C^1$ boundary $bD_0$ and $g\in C^1(\overline{D_0})$, then
\[
\sC(g)=\bm{C}\left(g_{\bar\zeta}\right)+g.
\]
\end{lemma}

\begin{proof} This appears in standard texts, often called the generalized Cauchy Integral formula. See for example \cite{hormander_scv_book}, Theorem 1.2.1; \cite{Krantz_scv_book}, Corollary 1.1.5; \cite{ChenShaw01}, Theorem 2.1.1; or \cite{varolin_book}, Theorem 1.1.2 for a proof.
\end{proof}
Here (and previously in the paper), the implied meaning is an equation holds functionally  when variables are not expressly notated.
The relation ``$\text{id}-\sC=-\bm{C}$" in Lemma \ref{L:stokes} yields a solution operator for $\dbar$ in the smooth category:

\begin{corollary}\label{L:C1solution} Suppose $D_0\subset\C$ is a smoothly bounded domain and $f\in C^{\infty}(\overline{D_0})$. Define
$v(z)=-\bm{C}(f)(z)$.

Then $\frac{\partial v}{\partial\bar z}=f$ and $v\in C^{\infty}(\overline{D_0})$.
\end{corollary}

\begin{proof}
Lemma \ref{L:smoothCT} shows that $v\in C^\infty\left(\overline{D_0}\right)$. It remains to show $\frac{\partial v}{\partial\bar z}= f$, which can be done locally.

Let $p\in D_0$ be arbitrary. Choose $\varphi\in C^\infty_0(D_0)$ such that $\varphi\equiv 1$ in a neighborhood $V$ of $p$. Split $\bm{C}(f)$ as
\begin{align*}
v(z)&=-\frac 1{2\pi i}\left[\int_{D_0}\frac{\varphi f}{\zeta -z}\, d\bar\zeta\wedge d\zeta +\int_{D_0}\frac{(1-\varphi) f}{\zeta -z}\, d\bar\zeta\wedge d\zeta\right] \\
&= I_1(z) +I_2(z).
\end{align*}
Since $(1-\varphi)\equiv 0$ in $V$, differentiation under the integral for $I_2$ shows $I_2\in\co(V)$. Thus $\frac{\partial v}{\partial\bar z}=\frac{\partial I_1}{\partial\bar z}$ in $V$.

However, $\varphi f$ has compact support, so the integral defining $I_1$ can be viewed as an integral over $\C$. Changing variables and differentiating under the integral sign
yields
\begin{equation*}
\frac{\partial v}{\partial\bar z}=\frac 1{2\pi i}\int_{\C} \frac{\frac{\partial(\varphi f)}{\partial\bar\zeta}(z-\zeta)}{\zeta} \, d\bar\zeta\wedge d\zeta
\end{equation*}
for $z\in V$. Now reverse the change of variables and apply Lemma \ref{L:stokes}. The result is
\begin{align*}
\frac{\partial v}{\partial\bar z}&= -\frac 1{2\pi i}\int_{\C} \frac{\frac{\partial(\varphi f)}{\partial\bar\zeta}(\zeta)}{\zeta-z} \, d\bar\zeta\wedge d\zeta 
= -\frac 1{2\pi i}\int_{D_0} \frac{\frac{\partial(\varphi f)}{\partial\bar\zeta}(\zeta)}{\zeta-z} \, d\bar\zeta\wedge d\zeta \\ &= \varphi\cdot f
\end{align*}
for $z\in V$. Thus $\frac{\partial v}{\partial\bar z}= f$ near $p$.
\end{proof}

\subsection{$L^p$ mapping}
The basic $L^1$ result on the Cauchy transform is

\begin{lemma}
\label{Lp1dim}
Let $D_0\subset\C$ be a bounded domain. If $g\in L^1(D_0)$, the function
\[
G(z)=\frac{-1}{2 \pi i}\int_{D_0}\frac{g(\zeta)}{\zeta-z}\,d\bar \zeta\wedge d\zeta
\]
belongs to $L^1(D_0)$. Moreover, $\|G\|_{L^1(D_0)}\leq C\|g\|_{L^1(D_0)}$, for $C>0$ independent of $g$.
\end{lemma}

\begin{proof}
For $S\subset\C$, let $\chi_S$ denote the characteristic function of $S$.

Set $\tilde g(\zeta)=g(\zeta)\chi_{D_0}(\zeta)$. Since $g\in L^1(D_0)$,  $\tilde g\in L^1(\C)$. Choose $R>\text{diam}(D_0)$.  If $B=B(0;R)$ is the disc centered at $0$ of radius $R$, let $h(\zeta)=\frac{1}{|\zeta|}\chi_{B}(\zeta)$.
Note $h\in L^1(\C)$.

By Young's inequality, $\tilde g * h\in L^1(\C)$. For any $z,\zeta\in D_0$, $|z-\zeta|\le\text{diam}(D_0)<R$, so $z-\zeta\in B$. Therefore,
\begin{align*}
\int_{D_0}\left|\int_{D_0}\frac{g(\zeta)\,dA(\zeta)}{\zeta-z}\right|\,dA(z)
&\le\int_{\C}\int_{\C}\frac{|\tilde g(\zeta)|\chi_B(\zeta-z)\,dA(\zeta)}{|\zeta-z|}\,dA(z)\\
&\le C\int_{D_0}|g(z)|\,dA(z).
\end{align*}
\end{proof}

A sharper result is

\begin{lemma}
\label{Linfty1dim}
Let $D_0\subset\C$ be a bounded domain. For $p\in[1,2)$ let $r=1$; for $p\in[2,\infty]$ let $r>2p/(p+2)$. 

If $g\in L^r(D_0)$, the function
\[
\left|\bm{C}\right| g (z)=: \frac{1}{2 \pi i}\int_{D_0}\frac{\left|g(\zeta)\right|}{|\zeta-z|}\,d\bar \zeta\wedge d\zeta
\]
belongs to $L^{p}(D_0)$.Moreover,  $\left\|\left|\bm{C}\right| g\right\|_{L^p(D_0)}\leq C\|g\|_{L^r(D_0)}$, for a constant $C>0$ independent of $g$.
\end{lemma}

\begin{proof}
The argument follows the proof of Lemma \ref{Lp1dim}, using Young's convolution inequality (see \cite[Example 4, page 60]{SteinAnalysis4}).
Choose $R>\text{diam}(D_0)$, let $B$ be the disc centered at $0$ of radius $R$. Set  $\tilde g(\zeta)=g(\zeta)\chi_{D_0}(\zeta)$. Then

\begin{align*}
\left(\int_{D_0}\left(\int_{D_0}\frac{|g(\zeta)|\,dA(\zeta)}{|\zeta-z|}\right)^p\,dA(z)\right)^{1/p}
&\le \left\| |\tilde g|*\frac{\chi_B}{|\zeta|}\right\|_{L^p(\C)}\\
&\le \|\tilde g\|_{L^r(\C)}\cdot\|\chi_B/|\zeta|\|_{L^{r'}(\C)},
\end{align*}
where $1/r+1/r'=1/p+1$, by Young's inequality. However with $r$ chosen as in the hypothesis,  it must hold that $r'<2$, for any $p\in[1,\infty]$. Consequently $\|\chi_B/|\zeta|\|_{L^{r'}(\C)}=C<\infty$, completing the proof.
\end{proof}

\bibliographystyle{acm}
\bibliography{ChenMcNeal18}

\end{document}